\documentclass[]{amsart}
\usepackage[OT2,T1]{fontenc}
\usepackage{amsmath}
\usepackage{amssymb}
\usepackage{yfonts}
\usepackage{mathtools}
\usepackage{txfonts}
\usepackage[mathscr]{euscript}
\usepackage{calrsfs}
\usepackage[dvips]{graphicx}
\newtheorem{theorem}{Theorem}
\newtheorem{corollary}{Corollary}
\newtheorem{definition}{Definition}

\newtheorem{lemma}{Lemma}

\newtheorem{remark}{Remark}
\newcommand{\eps}{\varepsilon}
\DeclareMathOperator{\RE}{Re}
\DeclareMathOperator{\Limsup}{Lim\;sup}
\DeclareMathOperator{\Graph}{Gr}
\DeclareMathOperator{\fix}{Fix}
\DeclareMathOperator{\F}{F}
\DeclareMathOperator{\uu}{U}
\DeclareMathOperator{\A}{A}

\DeclareMathOperator{\co}{co}
\DeclareMathOperator{\dist}{dist}

\DeclareMathOperator{\spn}{span}

\newcommand{\w}{\tilde}
\newcommand{\map}{\multimap}
\newcommand{\<}{\leqslant}

\newcommand{\n}{{n\geqslant 1}}
\newcommand{\K}{{k\geqslant 1}}

\newcommand{\z}[1]{(#1)_{n=1}^\infty}
\newcommand{\x}[1]{\{#1\}_{n=1}^\infty}
\newcommand{\R}[1]{\varmathbb{R}^{#1}}
\newcommand{\f}{\left}
\newcommand{\g}{\right}
\newcommand{\res}[2]{#1\:\rule[-1.5mm]{0.45pt}{4mm}\,\rule[-1mm]{0mm}{4mm}_{#2}}

\begin{document}
\title[Integrated solutions of non-densely defined integro-differential inclusions]{Integrated solutions of non-densely defined semilinear integro-differential inclusions: existence, topology and applications}
\author{Rados\l aw Pietkun}
\address{Toru\'n, Poland}
\email{rpietkun@pm.me}

\begin{abstract}
Given a linear closed but not necessarily densely defined operator $A$ on a Banach space $E$ with nonempty resolvent set and a multivalued map $F\colon I\times E\map E$ with weakly sequentially closed graph, we consider the integro-differential inclusion 
\begin{center}
$\dot{u}\in Au+F(t,\int u)\;\;\text{on }I,\;\;u(0)=x_0.$
\end{center}
We focus on the case when $A$ generates an integrated semigroup and obtain existence of integrated solutions in the sense of \cite[Def.6.4.]{thieme} if $E$ is weakly compactly generated and $F$ satisfies \[\beta(F(t,\Omega))\<\eta(t)\beta(\Omega)\;\;\text{for all bounded }\Omega\subset E,\] where $\eta\in L^1(I)$ and $\beta$ denotes the De Blasi measure of noncompactness. When $E$ is separable, we are able to show that the set of all integrated solutions is a compact $R_\delta$-subset of the space $C(I,E)$ endowed with the weak topology. We use this result to investigate a nonlocal Cauchy problem described by means of a nonconvex-valued boundary condition operator. Some applications to partial differential equations with multivalued terms are also included.
\end{abstract}

\maketitle

\section{Introduction and Notation}
The aim of this paper is the study of the following integro-differential inclusion in the Banach space~$E$:
\begin{equation}\label{P}
\begin{dcases*}
\dot{u}(t)\in Au(t)+F\f(t,\int_0^tu(s)\,ds\g)&on $I:=[0,T]$,\\
u(0)=x_0,
\end{dcases*}
\end{equation}
where $A\colon D(A)\subset E\to E$ is a linear closed operator, $F\colon I\times E\map E$ is a multivalued perturbation and $x_0\in E$ is given.\par The generic type of the semilinear differential inclusion, given by
\begin{equation}\label{2P}
\begin{dcases*}
\dot{u}(t)\in Au(t)+F(t,u(t))&on $I$,\\
u(0)=x_0,
\end{dcases*}
\end{equation}
with $A$ being the infinitesimal generator of a $0$-times integrated semigroup is thoroughly examined in the literature. The basic theory of methods with applications to semilinear differential inclusions of the type \eqref{2P} and methods of measures of non-compactness can be found in the monograph \cite{zecca}.\par As one knows, the domain $D(A)$ of a strongly continuous semigroup generator must be dense in $E$. However, a concept introduced in the eighties by Arendt (\cite{arendt}) allow to extend the theory to the case of abstract Cauchy problems with operators which do not satisfy the Hille-Yosida conditions. The main idea behind this notion can be summarized as follows: Let $\{U_t\}_{t\geqslant 0}$ be a $C_0$-semigroup on $E$. Then $S(t):=\int_0^tU(s)\,ds$ defines a family $\{S(t)\}_{t\geqslant 0}$ of bounded operators having the following three properties:
\begin{itemize}
\item[$($i$)$] $S(0)=0$,
\item[$($ii$)$] $t\mapsto S(t)$ is strongly continuous,
\item[$($iii$)$] $S(s)S(t)=\int_0^s(S(r+t)-S(t))\,dr$.
\end{itemize}
We call an {\it integrated semigroup} an operator family satisfying (i)-(iii) (for more information about defined notion, please refer to \cite{hieber, neu, thieme}). The generator $A$ of an integrated semigroup $\{S(t)\}_{t\geqslant 0}$ poses an example of such a linear operator that does not meet the Hille-Yosida conditions.\par In order to find a solution $u\in C(I,D(A))$, which is differentiable and satisfies 
\begin{equation}\label{P2}
\begin{dcases*}
\dot{u}(t)=Au(t)+f(t)&on $I$,\\
u(0)=x_0,
\end{dcases*}
\end{equation}
one usually must impose a lot of smoothness both on $x_0$ ($x_0\in D(A)$, $x_0\in D(A^2)$) and on $f$, either in the form of temporal regularity (i.e. $f\in W^{1,p}(I,E)$) or spatial regularity (i.e. $f(t)$ is supposed to belong to $D(A)$ a.e. on $I$). Without this additional regularity assumptions one considers problem \eqref{P2} in a generalized sense, which is suggested by the formal integration of both sides of \eqref{P2}. In this case we are dealing with {\it integral solutions} in the sense of Da Prato and Sinestrari (\cite{prato}): \[u(t)=x_0+A\int_0^tu(s)\,ds+\int_0^tf(s)\,ds,\;\;t\in I,\] which means in particular that $\int_0^tu(s)\,ds\in D(A)$. The authors of \cite{obuh} obtained solutions of the initial value problem \eqref{2P} in the latter sense under the following assumptions:
\begin{itemize}
\item[(i)] $A$ is the generator of a locally Lipschitz continuous non-degenerate exponentially bounded integrated semigroup $\{S(t)\}_{t\geqslant 0}$, whose derivative $\{S'(t)\}_{t\geqslant 0}$ forms an equicontinuous semigroup,
\item[(ii)] the set-valued perturbation term $F$ is a convex compact valued upper-Carath\'eodory multimap possessing the usual sublinear growth and condensing with respect to the Hausdorff measure of non-compactness.
\end{itemize}
One easily deduces that if an integral solution of \eqref{P2} exists then necessarily $x_0\in\overline{D(A)}$. If we want to relax smoothness condition for $x_0$ even more, we can integrate \eqref{P2} twice. The latter approach motivates the following definition (compare \cite[Def.6.4.]{thieme}): 
\begin{definition}
By an integrated solution of the problem \eqref{P} we mean a continuous function $u\colon I\to E$ such that 
\begin{equation}
\begin{dcases*}
\int_0^tu(s)\,ds\in D(A),\\
u(t)\in tx_0+A\!\!\int_0^tu(s)\,ds+\int_0^t(t-s)F(s,u(s))\,ds&for $t\in I$,
\end{dcases*}
\end{equation}
where the last integral on the right is understood in the sense of Aumann.
\end{definition}
\par The main results of our paper are theorems regarding the existence of integrated solutions of the problem \eqref{P} and the topological characterization of their set, in the situation where operator A is a generator of a non-degenerate exponentially bounded integrated semigroup and the multivalued perturbation term has weakly sequentially closed graph. To avoid compactness assumptions, the weak topology and the notion of the De Blasi measure of noncompactness is employed. Exploitation of Theorem 2.8. from \cite{kunze}, on the behaviour of the measure $\beta$ with respect to integration, allowed us to formulate the results also in the context of non-reflexive Banach spaces. \par In Section 2. we present some important, from the technical point of view, generalizations of the result known in the literature as the Convergence Theorem. Section 3. contains the aforesaid main results of the paper (Theorem \ref{existence}. and Theorem \ref{solset}.). Consequences of the previously described geometric structure of the set of integrated solutions to the Cauchy problem \eqref{P} has been collected in Section 4. in the form of theorems and examples illustrating the use of Theorem \ref{solset}.\par Let us introduce some notations which will be used in this paper.\par Let $(E,|\cdot|)$ be a Banach space, $E^*$ its normed dual and $\sigma(E,E^*)$ its weak topology. If $M$ is a subset of a Banach space $E$, by $(M,w)$ we denote the topological space $M$ furnished with the relative weak topology of $E$. \par The normed space of bounded linear operators $S\colon E\to E$ is denoted by $\mathcal{L}(E)$. Given $S\in\mathcal{L}(E)$, $||S||_{{\mathcal L}}$ is the norm of $S$. For any $\eps>0$ and $A\subset E$, $B(A,\eps)$ ($D(A,\eps)$) stands for an open (closed) $\eps$-neighbourhood of the set $A$ ($D_C(0,R)$ represents the ball in the space of continuous functions). If $x\in E$ we put $\dist(x,A):=\inf\{|x-y|\colon y\in A\}$. Besides, for two nonempty closed bounded subsets $A$, $B$ of $E$ the symbol $h(A,B)$ stands for the Hausdorff distance from $A$ to $B$, i.e. $h(A,B):=\max\{\sup\{\dist(x,B)\colon x\in A\},\sup\{\dist(y,A)\colon y\in B\}\}$.\par We use symbols of functional spaces, such as $C(I,E)$, $L^p(I,E)$, $L^\infty(I,E)$, $H^2(\R{n})$, $L^2(\R{n})$, in their commonly accepted meaning. Symbols $||\cdot||$, $||\cdot||_p$ represent norms in the space $C(I,E)$ and $L^p(I,E)$, respectively.\par Given metric space X, a set-valued map $F\colon X\map E$ assigns to any $x \in X$ a nonempty subset $F(x)\subset E$. $F$ is (weakly) upper semicontinuous, if the small inverse image $F^{-1}(A)=\{x\in X\colon F(x)\subset A\}$ is open in $X$ whenever $A$ is (weakly) open in $E$. We say that $F\colon X\map E$ is upper hemicontinuous if for each $x^*\in E^*$, the function $\sigma(x^*,F(\cdot))\colon X\to\R{}\cup\{+\infty\}$ is upper semicontinuous (as an extended real function), where $\sigma(x^*,F(x))=\sup\limits_{y\in F(x)}\langle x^*,y\rangle$. We have the following characterization: a map $F\colon X\map E$ with convex values is weakly upper semicontinues and has weakly compact values iff given a sequence $(x_n,y_n)$ in the graph $\Graph(F)$ of map $F$ with $x_n\xrightarrow[n\to\infty]{X}x$, there is a subsequence $y_{k_n}\xrightharpoonup[n\to\infty]{E}y\in F(x)$ ($\rightharpoonup$ denotes the weak convergence).\par Let $H^\ast(\cdot)$ denote the Alexander-Spanier cohomology functor with coefficients in the field of rational numbers ${\mathbb Q}$ (see \cite{spanier}). We say that a nonempty topological space $X$ is acyclic if the reduced cohomology $\w{H}^q(X)$ is $0$ for any $q\geqslant 0$. A nonempty compact metric space $X$ is an $R_\delta$-set if it is the intersection of a decreasing sequence of compact contractible metric spaces. In particular, $R_\delta$-sets are acyclic.\par An upper semicontinuous map $F\colon E\map E$ is called acyclic if it has compact acyclic values. A set-valued map $F\colon E\map E$ is admissible (in the sense of \cite[Def.40.1]{gorn}) if there is a Hausdorff topological space $\Gamma$ and two continuous functions $p,q\colon\Gamma\to E$ such that $F(x)=q(p^{-1}(x))$ for every $x\in E$ with $p$ being a surjective perfect map with acyclic fibers. Clearly, every acyclic map is admissible. Moreover, the composition of admissible maps is admissible (\cite[Th.40.6]{gorn}).
\par A real function $\beta$ defined on the family of bounded subsets $\Omega$ of $E$ defined by the formulae \[\beta(\Omega):=\inf\f\{\eps>0\colon\Omega\text{ has a weakly compact }\eps\text{-net in }E\vphantom{\int_0^t}\g\}\] is called the De Blasi measure of noncompactness. Recall that $\beta$ is a measure of noncompactness in the sense of general definition provided $E$ is endowed with the weak topology. One can readily verify that the MNC $\beta$ is regular, monotone, nonsingular, semi-additive, algebraically semi-additive and invariant under translation (see \cite{blasi}). \par We recall the reader following results on account of their practical importance. The first is a weak compactness criterion in $L^p(\Omega,E)$, which originates from \cite{ulger}.

\begin{theorem}[\protect{\cite[Cor.9]{ulger}}]\label{ulger}
Let $(\Omega,\Sigma,\mu)$ be a finite measure space with $\mu$ being a nonatomic measure on $\Sigma$. Let $A$ be a uniformly $p$-integrable subset of $L^p(\Omega,E)$ with $p\in[1,\infty)$. Assume that for a.a. $\omega\in\Omega$, the set $\{f(\omega)\colon f\in A\}$ is relatively weakly compact in $E$. Then $A$ is relatively weakly compact.
\end{theorem}

The next theorems are two well-known results from the scope of topological fixed point theory, our proofs could not do without.
\begin{theorem}[\protect{\cite[Th.7.4]{gorn2}}]\label{Lefschetz}
Let $X$ be an absolute extensor for the class of compact metrizable spaces and $F\colon X\map X$ be an admissible map such that $F(X)$ is contained in a compact metrizable subset of $X$. Then $F$ has a fixed point.
\end{theorem}

\begin{theorem}[\protect{\cite[Th.5.2.18]{hu}}]\label{KyFan}
If $M$ is a nonempty compact and convex subset of a locally convex space $E$ and $F\colon M\map M$ is a convex compact valued upper semicontinuous set-valued map, then $F$ has a fixed point.
\end{theorem}

\section{The Convergence Theorem}
In the case of upper hemicontinuous maps the following theorem is an analogon of a relation binding usc set-valued maps and semi limits (cf. \cite[Prop.1.4.7]{aubin}).
\begin{theorem}\label{Plis}
Let $F\colon E\map E$ be a closed convex valued upper hemicontinuous multimap. Then
\begin{equation}\label{Plis:1}
\f(y\in\bigcap_{\eps>0}\bigcap_{\delta>0}\overline{\co}\bigcup_{x'\in B(x,\delta)}B(F(x'),\eps)\g)\Longleftrightarrow (x,y)\in\Graph(F).
\end{equation}
\end{theorem}

\begin{proof}
The ''only if'' part is basically obvious. It follows from the fact that $(x,y)\in\overline{\Graph(F)}$ if and only if $y\in\underset{x'\to x}{\Limsup}\,F(x')$, where the latter is the upper limit in the sense of Painlev\'e-Kuratowski. This limit is evidently contained in $\bigcap\limits_{\eps>0}\bigcap\limits_{\delta>0}\overline{\co}\bigcup\limits_{x'\in B(x,\delta)}B(F(x'),\eps)$.\par Fix $y\in\bigcap\limits_{\eps>0}\bigcap\limits_{\delta>0}\overline{\co}\bigcup\limits_{x'\in B(x,\delta)}B(F(x'),\eps)$. Let $x^*\in E^*$. By the definition of upper hemicontinuity \[\forall\,\eps>0\;\exists\,\delta>0\;\;\;\sigma(x^*,F(B(x,\delta)))<\sigma(x^*,F(x))+\eps.\] Thus, \[\inf_{\eps>0}\inf_{\delta>0}\sigma(x^*,F(B(x,\delta)))\<\inf_{\eps>0}\f(\sigma(x^*,F(x))+\eps\g).\] The latter property implies
\begin{align*}
\langle x^*,y\rangle&\<\sigma\f(x^*,\bigcap_{\eps>0}\bigcap_{\delta>0}\overline{\co}\bigcup_{x'\in B(x,\delta)}B(F(x'),\eps)\g)\<\inf_{\eps>0}\inf_{\delta>0}\sigma(x^*,\overline{\co}F(B(x,\delta)))\\&=\inf_{\eps>0}\inf_{\delta>0}\sigma(x^*,F(B(x,\delta)))\<\inf_{\eps>0}\f(\sigma(x^*,F(x))+\eps\g)\<\sigma(x^*,F(x))+|x^*|\inf_{\eps>0}\eps\\&=\sigma(x^*,F(x)).
\end{align*}
Since $F$ has closed convex values, it means that $y\in F(x)$, i.e. $(x,y)\in\Graph(F)$.
\end{proof}

Let $(I,{\mathscr L}(I),\ell)$ denote the Lebesgue measure space. The following property of upper hemicontinuous multimaps with closed and convex values is a key, although strictly technical, tool used in the proofs of results regarding differential inclusions.
\begin{corollary}[{\bf Pli\'s Convergence Theorem}]\label{convth}
Let $F\colon E\map E$ be a closed convex valued upper hemicontinuous multimap. Assume that functions $f_n,f\colon I\to E$ and $g_n,g\colon I\to E$ are such that 
\begin{equation}\label{Plis:2}
g_n(t)\xrightarrow[n\to\infty]{E}g(t)\;\text{ a.e. on }I
\end{equation}
and
\begin{equation}\label{Plis:3}
f_n(t)\in\overline{\co}B(F(B(g_n(t),\eps_n)),\eps_n)\;\text{ a.e. on }I, \text{where }\eps_n\to 0^+\text{ as }n\to\infty.
\end{equation}
If one of the following conditions holds
\begin{itemize}
\item[(i)] $f_n\xrightharpoonup[n\to\infty]{L^1(I,E)}f$,
\item[(ii)] $f_n$ and $f$ are weakly $\ell$-measurable and \[\forall\,J\in{\mathscr L}(I)\;\;\;(\mathrm{D})\int_Jf_n\,d\ell\xrightharpoonup[n\to\infty]{*}(\mathrm{D})\int_Jfd\ell,\] where $(\mathrm{D})\int_Jfd\ell$ is the Dunford integral of $f$ over $J$,
\item[(iii)] $f_n(t)\xrightharpoonup[n\to\infty]{E}f(t)$ a.e. on $I$,
\item[(iv)] $f(t)\in\overline{\co}\,$\mbox{w}-$\limsup\limits_{n\to\infty}\,\{f_n(t)\}$ a.e. on $I$, where \[\mbox{w}\text{-}\limsup\limits_{n\to\infty}\,A_n:=\f\{x\in E\colon x=\mbox{w}\text{-}\lim_{n\to\infty}x_{k_n}, x_{k_n}\in A_{k_n}, k_1<k_2<\ldots\g\}\,\text{ and }\,\x{A_n}\subset 2^E\setminus\{\varnothing\},\]
\item[(v)] $f(t)\in\bigcap\limits_{n=1}^\infty\overline{\co}\bigcup\limits_{m=n}^\infty\{f_m(t)\}$ a.e. on $I$,
\end{itemize}
then $f(t)\in F(g(t))$ a.e. on $I$.
\end{corollary}

\begin{proof}
The thesis can be inferred directly from the assumption (i), as it has been done many times in the past (cf. classic reference \cite{aubin2}). The implication between assumption (i) and (ii) can be easily justified under the additional assumption that $E^*$ has RNP. Convergence $f_n\xrightharpoonup[n\to\infty]{L^1(I,E)}f$ means that for every $g\in L^\infty(I,E^*)$, $\int_I\langle g(t),f_n(t)\rangle\,dt\xrightarrow[n\to\infty]{}\int_I\langle g(t),f(t)\rangle\,dt$. For every $J\in{\mathscr L}(I)$ and $x^*\in E^*$ define $g:=x^*{\bf 1}_J\in L^\infty(I,E^*)$. Then $\langle x^*,(\mathrm{D})\int_Jf_nd\ell\rangle\xrightarrow[n\to\infty]{}\langle x^*,(\mathrm{D})\int_Jfd\ell\rangle$. Consequently, $(\mathrm{D})\int_Jf_n\,d\ell\xrightharpoonup[n\to\infty]{*}(\mathrm{D})\int_Jfd\ell$.\par Condition (ii) entails condition (v). Assume that there is $n_0\in\mathbb{N}$, $x_0^*\in E^*$ and a subset $J\in{\mathscr L}(I)$ such that $\ell(J)>0$ and $\langle x_0^*,f(t)\rangle>\sup\limits_{m\geqslant n_0}\,\langle x_0^*,f_m(t)\rangle$ for every $t\in J$. The set $J$ has the form of a countable union of sets \[J_k:=\f\{t\in J\colon\langle x_0^*,f(t)\rangle>\sup\limits_{m\geqslant n_0}\,\langle x_0^*,f_m(t)\rangle+1/k\g\}.\] The sets $J_k$ are clearly measurable, since the function $I\ni t\mapsto\langle x_0^*,f(t)\rangle-\sup\limits_{m\geqslant n_0}\,\langle x_0^*,f_m(t)\rangle-1/k\in\mathbb{R}$ is $\ell$-measurable. Moreover, there must be a set $J_{k_0}$ such that $\ell(J_{k_0})>0$. Now, observe that
\begin{align*}
\f\langle x_0^*,(\mathrm{D})\int_{J_{k_0}}fd\ell\g\rangle&=\int_{J_{k_0}}\langle x_0^*,f(t)\rangle\,dt>\int_{J_{k_0}}\langle x_0^*,f_m(t)\rangle\,dt+\frac{\ell(J_{k_0})}{k_0}\\&=\f\langle x_0^*,(\mathrm{D})\int_{J_{k_0}}f_m\,d\ell\g\rangle+\frac{\ell(J_{k_0})}{k_0}
\end{align*}
for every $m\geqslant n_0$. In view of (ii) we have \[\f\langle x_0^*,(\mathrm{D})\int_{J_{k_0}}fd\ell\g\rangle\geqslant\f\langle x_0^*,(\mathrm{D})\int_{J_{k_0}}fd\ell\g\rangle+\frac{\ell(J_{k_0})}{k_0}\] - a contradiction. Thus, \[\forall\,\n\;\forall\,x^*\in E^*\;\;\;\langle x^*,f(t)\rangle\<\sup_{m\geqslant n}\,\langle x^*,f_m(t)\rangle\;\text{ a.e. on }I,\] i.e. $f(t)\in\bigcap\limits_{n=1}^\infty\overline{\co}\bigcup\limits_{m=n}^\infty\{f_m(t)\}$ a.e. on $I$.\par Of course, (iii) implies (iv) and (iv) implies (v).\par Fix $t\in I$ such that \eqref{Plis:2}, \eqref{Plis:3} and (v) are satisfied simultaneously. Take $\eps>0$ and $\delta>0$. In view of \eqref{Plis:2} there is $n\in\mathbb{N}$ such that $B(g_m(t),\eps_m)\subset B(g(t),\delta)$ and $\eps_m<\eps$ for $m\geqslant n$. From \eqref{Plis:3} it follows that 
\begin{align*}
\overline{\co}\bigcup_{m=n}^\infty\{f_m(t)\}&\subset\overline{\co}\bigcup_{m=n}^\infty\overline{\co}B(F(B(g_m(t),\eps_m)),\eps_m)\subset\overline{\co}\bigcup_{m=n}^\infty B(F(B(g_m(t),\eps_m)),\eps_m)\\&\subset\overline{\co}B(F(B(g(t),\delta)),\eps).
\end{align*}
Hence, \[f(t)\in\bigcap\limits_{n=1}^\infty\overline{\co}\bigcup\limits_{m=n}^\infty\{f_m(t)\}\subset\bigcap_{\eps>0}\bigcap_{\delta>0}\overline{\co}\bigcup_{x\in B(g(t),\delta)}B(F(x),\eps).\] Applying Theorem \ref{Plis}. one sees that $f(t)\in F(g(t))$.
\end{proof}

\begin{corollary}\label{wuhc}
Let $F\colon E\map E$ be a closed convex valued multimap satisfying:
\begin{equation}\label{wuhc2}
x_n\xrightharpoonup[n\to\infty]{E}x\Longrightarrow\limsup_{n\to\infty}\sigma(x^*,F(x_n))\<\sigma(x^*,F(x))\;\text{for all }x^*\in E^*.
\end{equation}
Assume that functions $f_n,f\colon I\to E$ and $g_n,g\colon I\to E$ are such that 
\begin{equation}\label{Plis:4}
g_n(t)\xrightharpoonup[n\to\infty]{E}g(t)\;\text{ a.e. on }I
\end{equation}
and
\begin{equation}\label{Plis:5}
f_n(t)\in\overline{\co}B(F(g_n(t)),\eps_n)\;\text{ a.e. on }I, \text{where }\eps_n\to 0^+\text{ as }n\to\infty.
\end{equation}
If the following condition holds
\begin{equation}\label{Plis:6}
f(t)\in\bigcap\limits_{n=1}^\infty\overline{\co}\bigcup\limits_{m=n}^\infty\{f_m(t)\}\;\;\text{a.e. on }I,
\end{equation}
then $f(t)\in F(g(t))$ a.e. on $I$.
\end{corollary}

\begin{proof}
Let $x_n\xrightharpoonup[n\to\infty]{E}x$. Then $x_n\xrightharpoonup[n\geqslant N]{E}x$ and $\limsup\limits_{n\geqslant N}\sigma(x^*,F(x_n))\<\sigma(x^*,F(x))$ for every $x^*\in E^*$ and $N\geqslant 1$, in view of \eqref{wuhc2}. Therefore, 
\begin{equation}\label{wuhc3}
\forall\,x^*\in E^*\;\;\sup_{N\geqslant 1}\inf_{n\geqslant N}\sup_{m\geqslant n}\sigma(x^*,F(x_m))\<\sigma(x^*,F(x)).
\end{equation}
Take $\eps>0$. There is $N\in\mathbb{N}$ such that $\eps_m<\eps$ for $m\geqslant N$. From \eqref{Plis:5} it follows that 
\[\overline{\co}\bigcup_{m=N}^\infty\{f_m(t)\}\subset\overline{\co}\bigcup_{m=N}^\infty\overline{\co}B(F(g_m(t)),\eps_m)\subset\overline{\co}\bigcup_{m=N}^\infty B(F(g_m(t)),\eps).\] Hence, \[f(t)\in\bigcap\limits_{n=1}^\infty\overline{\co}\bigcup\limits_{m=n}^\infty\{f_m(t)\}\subset
\bigcap_{n\geqslant N}\overline{\co}\bigcup_{m=n}^\infty B(F(g_m(t)),\eps)\] and eventually \[f(t)\in\bigcap_{\eps>0}\bigcup_{N=1}^\infty\bigcap_{n=N}^\infty\overline{\co}\bigcup_{m=n}^\infty B(F(g_m(t)),\eps)\;\;\text{a.e. on }I.\] Take $x^*\in E^*$. By \eqref{Plis:4} and \eqref{wuhc3} it follows that
\begin{align*}
\langle x^*,f(t)\rangle&\<\sigma\f(x^*,\bigcap_{\eps>0}\bigcup_{N=1}^\infty\bigcap_{n=N}^\infty\overline{\co}\bigcup_{m=n}^\infty B(F(g_m(t)),\eps)\g)\<\inf_{\eps>0}\sup_{N\geqslant 1}\inf_{n\geqslant N}\sup_{m\geqslant n}\sigma(x^*,B(F(g_m(t)),\eps))\\&\<\sup_{N\geqslant 1}\inf_{n\geqslant N}\sup_{m\geqslant n}\sigma(x^*,F(g_m(t)))+\inf_{\eps>0}\eps|x^*|\<\sigma(x^*,F(g(t))).
\end{align*}
Consequently, $f(t)\in F(g(t))$ a.e. on $I$.
\end{proof}

\section{Existence and topology of solutions}
The remainder of the article rests on the following hypotheses:
\begin{itemize}
\item[$(\A_1)$] $A\colon D(A)\to E$ is a generator of a non-degenerate integrated semigroup $\{S(t)\}_{t\geqslant 0}$ such that $||S(t)||\<Me^{\omega t}$ for $t\geqslant 0$ with suitable constants $M>0$ and $\omega\in\R{}$,
\item[$(\A_2)$] $A\colon D(A)\to E$ satisfies $(\A_1)$ and the generated semigroup $\{S(t)\}_{t\geqslant 0}$ is equicontinuous,
\item[$(\F_1)$] for every $(t,x)\in I\times E$ the set $F(t,x)$ is nonempty and convex,
\item[$(\F_2)$] the map $F(\cdot,x)$ has a strongly measurable selection for every $x\in E$,
\item[$(\F_3)$] the graph $\Graph\f(F(t,\cdot)\g)$ is sequentially closed in $(E,w)\times(E,w)$ for a.a. $t\in I$,
%\item[$(\F_4)$] $F$ satisfies a sublinear growth condition, i.e. there is $b\in L^1(I,\R{})$ such that for all $x\in E$ and for a.a. $t\in I$,\[||F(t,x)||^+:=\sup\{|y|\colon y\in F(t,x)\}\<b(t)(1+|x|),\]
\item[$(\F_4)$] $F$ satisfies the following growth condition: \[\limsup_{r\to+\infty}r^{-1}\overline{\int\limits_I}\sup_{|x|\<r}||F(t,x)||^+\,dt<M^{-1}e^{-\omega T},\] where $M,\omega$ are exactly the same constants as in $(\A_1)$,
\item[$(\F_5)$] there is a function $\eta\in L^1(I,\R{})$ such that for all bounded $\Omega$ in $E$ and for a.a. $t\in I$ the inequality holds \[\beta(F(t,\Omega))\<\eta(t)\beta(\Omega).\]
\end{itemize}

\begin{remark}
A linear operator $A$ is called a generator of an integrated semigroup, if there exists $\omega\in\R{}$ such that $(\omega,\infty)\subset\rho(A)$, and there exists a strongly continuous exponentially bounded family $\{S(t)\}_{t\geqslant 0}$ of bounded operators such that $S(0)=0$ and $(\lambdaup-A)^{-1}=\lambdaup\int_0^\infty e^{-\lambdaup t}S(t)\,dt$ for $\lambdaup>\omega$. An integrated semigroup $\{S(t)\}_{t\geqslant 0}$ is called non-degenerate if $\bigcap\limits_{t\geqslant 0}\ker S(t)=\{0\}$.
\end{remark}

\begin{remark}
By condition $(\F_4)$ we mean implicitly that the map $I\ni t\mapsto\sup\limits_{|x|\<r}||F(t,x)||^+\in\R{}_+$ is bounded for every $r>0$. The upper integral of a bounded $($but not necessarily measurable$)$ function $f\colon I\to\R{}_+$ is \[\overline{\int\limits_I}f(t)\,dt:=\inf\f\{\int\limits_Ig(t)\,dt\colon g\in L^1(I), f(t)\<g(t)\text{ a.e. on }I\g\}.\]
\end{remark}

\begin{remark}
If the integrated semigroup $\{S(t)\}_{t\geqslant 0}$ is exponentially stable in the sense that $||S(t)||\<e^{-\omega t}$ for $t\geqslant 0$ with $\omega>0$, then assumption $(\F_4)$ shall take the form \[\limsup_{r\to+\infty}r^{-1}\overline{\int\limits_I}\sup_{|x|\<r}||F(t,x)||^+\,dt<1.\]
\end{remark}

Let $N_F\colon C(I,E)\map L^1(I,E)$ be the Nemtyski\v{\i} operator corresponding to $F$, i.e.
\[N_F(u):=\f\{w\in L^1(I,E)\colon w(t)\in F(t,u(t))\mbox{ for a.a. }t\in I\g\}.\] 

\begin{remark}
Under hypotheses $(\F_1)$-$(\F_5)$ the Nemytski\v{\i} operator $N_F$ is nonempty convex weakly compact valued weakly upper semicontinuous set-valued map $($cf. for instance \cite[Prop.1.]{pietkun}$)$.
\end{remark}

Let us also define the Volterra integral operator $V\colon L^1(I,E)\to C(I,E)$ by the formulae:
\begin{equation}\label{volterra}
V(f)(t):=\int_0^tS(t-s)f(s)\,ds\;\;\text{for }t\in I.
\end{equation}

\begin{lemma}\label{mono}
Assume $(\A_1)$. Then the integral operator $V\colon L^1(I,E)\to C(I,E)$ defined by \eqref{volterra} is a bounded linear monomorphism with $||V||_{{\mathcal L}(L^1,C)}\<Me^{\omega T}$.
\end{lemma}

\begin{proof}
Apply \cite[Th.6.5.]{thieme} and the very definition of an integrated solution to the inhomogeneous Cauchy problem \eqref{P2}. The estimate of the norm $||V||_{{\mathcal L}(L^1,C)}$ follows straightforwardly.
\end{proof}

\par It is important from a methodological point of view to realize that the solution set $S_{\!F}(x_0)$ of all integrated solutions to the problem \eqref{P} coincides with the fixed point set $\fix(H)$ of the operator $H\colon C(I,E)\map C(I,E)$ defined by $H:=S(\cdot)x_0+V\circ N_F$. Indeed, if $u\in\fix(H)$ then $u=S(\cdot)x_0+V(f)$ for some $f\in N_F(u)$. Thanks to \cite[Th.6.5.]{thieme} we know that $u$ belongs to $S_{\!F}(x_0)$. Suppose then that $u\in S_{\!F}(x_0)$. This means that \[u(t)=tx_0+A\int_0^tu(s)\,ds+\int_0^t(t-s)f(s)\,ds\] for some $f\in N_F(u)$. The inhomogeneous Cauchy problem \eqref{P2} has a unique integrated solution $x$ given by the formula: $x=S(\cdot)x_0+V(f)$. This follows again from the use of \cite[Th.6.5.]{thieme}. Since $u$ is also a solution to \eqref{P2}, it means that $u=S(\cdot)x_0+V(f)$, i.e. $u\in\fix(H)$.\par Recall that the space $E$ is called {\it weakly compactly generated} (WCG) if there is a weakly compact set $K$ in $E$ such that $E=\overline{\spn}(K)$.

\begin{lemma}\label{lemat1}
Let $E$ be a WCG space. Assume $(\A_2)$, $(\F_1)$ and $(\F_3)$-$(\F_5)$. Then the solution set $S_{\!F}(x_0)$ is weakly compact in $C(I,E)$.
\end{lemma}

\begin{proof}
We claim that there are a priori bounds for $S_{\!F}(x_0)$. Indeed, assume that for every $\n$ there exists $x_n\in S_{\!F}(x_0)$ such that $|x_n|>n$. Let $x_n=S(\cdot)x_0+V(f_n)$ for some $f_n\in N_F(x_n)$. By $(\F_4)$ we have
\begin{align*}
1\<\overline{\lim_{n\to\infty}}\,\frac{||x_{n}||}{n}&\<\overline{\lim_{n\to\infty}}\,\frac{||S(\cdot)x_0+V(f_{n})||}{n}\<\overline{\lim_{n\to\infty}}\,\frac{Me^{\omega T}|x_0|+||V||_{\mathcal L}||f_{n}||_1}{n}\\&\<\lim_{n\to\infty}\,\frac{Me^{\omega T}|x_0|}{||x_{n}||}+Me^{\omega T}\overline{\lim_{n\to\infty}}\,\frac{\overline{\int\limits_I}\sup\limits_{|x|\<||x_{n}||}||F(t,x)||^+\,dt}{||x_{n}||}<1.
\end{align*}
Hence, the claim is validated.\par The solution set $S_{\!F}(x_0)$ is strongly equicontinuous in $C(I,E)$. Indeed, take an arbitrary $u\in S_{\!F}(x_0)$. Then $u=S(\cdot)x_0+V(f)$ for some $f\in N_F(u)$. Let $g\in L^1(I)$ be such that $\sup\limits_{|x|\<||S_{\!F}(x_0)||^+}||F(t,x)||^+\<g(t)$ a.e. on $I$. As one can see
\begin{equation}\label{equicontinuity}
\begin{split}
|u(t)-u(\tau)|&\<|S(t)x_0-S(\tau)x_0|+\f|\int_0^tS(t-s)f(s)\,ds-\int_0^\tau S(\tau-s)f(s)\,ds\g|\\&\<|S(t)x_0-S(\tau)x_0|+\int_0^\tau||S(t-s)-S(\tau-s)||_{\mathcal L}g(s)\,ds+Me^{\omega T}\int_\tau^tg(s)\,ds
\end{split}
\end{equation}
From Lebesgue's dominated convergence theorem and assumption $(\A_2)$ it follows that \[\lim_{t\to\tau}\sup_{u\in S_{\!F}(x_0)}|u(t)-u(\tau)|=0.\] \par Now, choose an arbitrary $\z{x_n}\subset S_{\!F}(x_0)$. Let $f_n\in N_F(x_n)$ be such that $x_n=S(\cdot)x_0+V(f_n)$. It is easy to see that the strong equicontinuity of $\x{x_n}$ implies continuity of the function $I\ni t\mapsto\beta(\z{x_n(t)})\in\R{}_+$ (compare \cite{szufla}). From \cite[Th.2.8.]{kunze} it follows that
\begin{align*}
\beta(\x{x_n(t)})&=\beta\f(\f\{S(t)x_0+V(f_n)(t)\g\}_{n=1}^\infty\g)\<\beta\f(\f\{\int_0^tS(t-s)f_n(s)\,ds\g\}_{n=1}^\infty\g)\\&\<\int_0^t||S(t-s)||_{\mathcal L}\beta(\x{f_n(s)})\,ds\<Me^{\omega T}\int_0^t\eta(s)\beta(\x{x_n(s)})\,ds
\end{align*}
for $t\in I$. By Gronwall's inequality, $\beta(\x{x_n(t)})=0$ for every $t\in I$. In particular, $\beta(\x{f_n(t)})=0$ a.e. on $I$. The family $\x{f_n}$ is uniformly integrable, since 
\begin{equation}\label{unifint}
\lim_{\ell(J)\to 0}\sup_\n\int\limits_J|f_n(t)|\,dt\<\lim_{\ell(J)\to 0}\overline{\int\limits_J}\sup_{|x|\<||S_{\!F}(x_0)||^+}||F(t,x)||^+\,dt\<\lim_{\ell(J)\to 0}\int\limits_J g(t)\,dt=0
\end{equation}
for some $g\in L^1(I)$. In view of Theorem \ref{ulger}. we can extract a subsequence, again denoted by, $\z{f_n}$ such that $f_n\xrightharpoonup[n\to\infty]{L^1(I,E)}f$.\par Observe that conditions $(\F_1)$, $(\F_5)$ together with the hypothesis regarding $w$-$w$ sequential closedness of $\Graph(F(t,\cdot))$ implies \eqref{wuhc2}. Put $x:=S(\cdot)x_0+V(f)=w$-$\lim\limits_{n\to\infty}S(\cdot)x_0+V(f_n)$. Then $x_n\xrightharpoonup[n\to\infty]{C(I,E)}x$. In particular, $x_n(t)\xrightharpoonup[n\to\infty]{E}x(t)$ for each $t\in I$. Therefore, assumptions \eqref{Plis:4}, \eqref{Plis:5} and \eqref{Plis:6} of Corollary \ref{wuhc}. are satisfied (the implication $f_n\xrightharpoonup[n\to\infty]{L^1(I,E)}f\Rightarrow$ \eqref{Plis:6} was proved earlier). Consequently, $f\in N_F(x)$ and $x\in S(\cdot)x_0+V\circ N_F(x)$, i.e. $x\in S_{\!F}(x_0)$.
\end{proof}

The main result regarding the existence of integrated solutions to the initial value problem \eqref{P} is contained in the following:
\begin{theorem}\label{existence}
Let $E$ be a WCG space. Assume that hypotheses $(\A_1)$, $(\F_1)$-$(\F_5)$ are satisfied. Then the solution set $S_{\!F}(x_0)$ of the Cauchy problem \eqref{P} is nonempty.
\end{theorem}

\begin{proof}
Assume that $u_n\xrightharpoonup[n\to\infty]{C(I,E)}u$. It means in particular that $\sup\limits_{t\in I}\beta(\x{u_n(t)})=0$. Let $f_n\in N_F(u_n)$ for $\n$. Observe that the family $\x{f_n}$ is uniformly integrable, since the set $\x{u_n}$ is bounded. The latter follows by analogy to \eqref{unifint}. Taking into account the fact that $\beta(\x{f_n(t)})\<\eta(t)\beta(\x{u_n(t)})=0$ a.e. on $I$, we infer that $f_n\xrightharpoonup[n\to\infty]{L^1(I,E)}f$, passing to a subsequence if necessary. This results from Theorem \ref{ulger}. Moreover, assumptions of Corollary \ref{wuhc}. are satisfied and conclusion that $f\in N_F(u)$ follows. Since $S(\cdot)x_0+V(f_n)\in H(u_n)$ and $S(\cdot)x_0+V(f_n)\xrightharpoonup[n\to\infty]{C(I,E)}S(\cdot)x_0+V(f)\in H(u)$, we conclude that the restriction $H\colon(M,w)\map(C(I,E),w)$ is a convex compact valued upper semicontinuous map for every weakly compact $M\subset C(I,E)$. \par It is easy to show that the operator $H$ possesses an invariant ball $D_C(0,R)$. Assume to the contrary that for any $\n$ there exist $||u_n||\<n$ and $v_n\in H(u_n)$ such that $||v_n||>n$. Then
\begin{align*}
1\<\overline{\lim_{n\to\infty}}\,\frac{||v_n||}{n}&\<\overline{\lim_{n\to\infty}}\,\frac{||S(\cdot)x_0+V(N_F(u_n))||^+}{n}\<\overline{\lim_{n\to\infty}}\,\frac{Me^{\omega T}|x_0|+||V||_{\mathcal L}||N_F(u_n)||_1^+}{n}\\&\<\lim_{n\to\infty}\,\frac{Me^{\omega T}|x_0|}{n}+Me^{\omega T}\overline{\lim_{n\to\infty}}\,\frac{\overline{\int\limits_I}\sup\limits_{|x|\<n}||F(t,x)||^+\,dt}{||n||}<1,
\end{align*}
by $(\F_4)$.\par Assume that the radius $R>0$ is such that $H(D_C(0,R))\subset D_C(0,R)$. Fix $\hat{x}\in D_C(0,R)$ and define \[{\mathcal A}:=\f\{M\in 2^{D_C(0,R)}\setminus\{\varnothing\}\colon M\text{ is closed convex and } \overline{\co}\f(\{\hat{x}\}\cup H(M)\g)\subset M\g\}.\] Then the intersection $M_0:=\bigcap\limits_{M\in{\mathcal A}}M$ is nonempty ($D_C(0,R)\in{\mathcal A}$) and \[\overline{\co}\f(\{\hat{x}\}\cup H(M_0)\g)\subset\bigcap\limits_{M\in{\mathcal A}}\overline{\co}\f(\{\hat{x}\}\cup H(M)\g)\subset M_0.\] Since $\overline{\co}\f(\{\hat{x}\}\cup H(M_0)\g)\in{\mathcal A}$, we have $M_0\subset\overline{\co}\f(\{\hat{x}\}\cup H(M_0)\g)$. From that it follows the equality, $M_0=\overline{\co}\f(\{\hat{x}\}\cup H(M_0)\g)$.\par We claim that $M_0$ is weakly compact in $C(I,E)$. Since
\begin{equation}\label{phi(L)}
\varphi(L):=\sup_{t\in I}e^{-Lt}\int_0^t e^{Ls}\eta(s)\,ds\xrightarrow[L\to+\infty]{}0,
\end{equation}
we can always pick a constant $L_0>0$ so that $Me^{\omega T}\varphi(L_0)<1$. Let $\beta_{L_0}$ be a set function defined on the family of all bounded subsets of $C(I,E)$, given by the formulae: \[\beta_{L_0}(M):=\max\f\{\sup_{t\in I}e^{-L_0t}\beta(D(t))\colon D\subset M\;\text{denumerable}\g\}.\] Clearly, $\beta_{L_0}$ is a nonsingular measure of noncompactness on $C(I,E)$. Therefore, one can always choose a subset $\x{u_n}\subset H(M_0)$ in such a way that \[\sup_{t\in I}e^{-L_0t}\beta\f(\x{u_n(t)}\g)=\beta_{L_0}(H(M_0))=\beta_{L_0}(M_0).\] Let $u_n=S(\cdot)x_0+V(f_n)$ with $f_n\in N_F(v_n)$ and $v_n\in M_0$ for $\n$. Making use of \cite[Th.2.8.]{kunze} we can derive the following estimation
\begin{align*}
\sup_{t\in I}e^{-L_0t}\beta\f(\x{u_n(t)}\g)&=\sup_{t\in I}e^{-L_0t}\beta\f(\x{S(t)x_0+V(f_n)(t)}\g)\\&\<\sup_{t\in I}e^{-L_0t}\int_0^t||S(t-s)||_{\mathcal L}\beta(\x{f_n(s)})\,ds\\&\<\sup_{t\in I}e^{-L_0t}Me^{\omega T}\int_0^t\eta(s)\beta(\x{v_n(s)})\,ds\\&\<Me^{\omega T}\f(\sup_{t\in I}e^{-L_0t}\int_0^te^{L_0s}\eta(s)\,ds\g)\sup_{t\in I}e^{-L_0t}\beta\f(\x{v_n(t)}\g)\\&\<Me^{\omega T}\varphi(L_0)\sup_{t\in I}e^{-L_0t}\beta\f(\x{u_n(t)}\g).
\end{align*}
Whence, $\Gamma:=\sup\limits_{t\in I}e^{-L_0t}\beta\f(\x{u_n(t)}\g)=0$. Consider an arbitrary sequence $\z{x_n}\subset H(M_0)$. For each $\n$, there is $v_n\in M_0$ and $w_n\in N_F(v_n)$ such that $x_n=S(\cdot)x_0+V(w_n)$. Condition $(\F_5)$ implies \[e^{-L_0t}\beta(\x{w_n(t)})\<e^{-L_0t}\eta(t)\beta(\x{v_n(t)})\<\eta(t)\sup_{t\in I}e^{-L_0t}\beta(\x{v_n(t)})\<\eta(t)\Gamma.\] Thus, $\beta(\x{w_n(t)})=0$ a.e. on $I$. At the same time $\x{w_n}$ is uniformly integrable. Consequently, we may assume, passing to a subsequence if necessary, that $w_n\xrightharpoonup[n\to\infty]{L^1(I,E)}~w$. The latter entails, $x_n=S(\cdot)x_0+V(w_n)\xrightharpoonup[n\to\infty]{C(I,E)}S(\cdot)x_0+V(w)$, i.e. the set $H(M_0)$ is relatively weakly compact. The weak compactness of $\overline{\co}\f(\{\hat{x}\}\cup H(M_0)\g)$ follows by the Kre{\v\i}n-Smulian theorem. Therefore, $M_0$ is weakly compact.\par Summing up, by virtue of Theorem \ref{KyFan}. we infer that the convex compact valued upper semicontinuous set-valued map $H\colon(M_0,w)\map(M_0,w)$ possesses at least one fixed point. This fixed point constitutes a solution to the Cauchy problem \eqref{P}.
\end{proof}

\begin{theorem}[Leray-Schauder continuation theorem for weak topologies]\label{weak}\mbox{}\\
Let $E$ be a weakly normal $w^*$-separable Banach space and $U\subset E$ weakly open. Assume that $F\colon(\overline{U}^w,w)\map(E,w)$ is sequentially upper semicontinuous with weakly compact convex values such that 
\begin{equation}\label{rwc2}
\bigcup\limits_{\lambdaup\in[0,1]}\f\{x\in\overline{U}^w\colon x\in(1-\lambdaup)x_0+\lambdaup F(x)\g\}\;\text{ is rwc}.
\end{equation}
Assume also that there exists $x_0\in U$ such that 
\begin{equation}\label{monch}
M\subset\overline{U}^w,\;M\subset\overline{\co}\,(\{x_0\}\cup F(M))\Rightarrow M\text{ is rwc}
\end{equation}
 and that 
\begin{equation}\label{L-S}
\forall\,x\in\overline{U}^w\setminus U\,\forall\,\lambdaup\in[0,1]\;\;\;x\not\in(1-\lambdaup)x_0+\lambdaup F(x).
\end{equation}
Then $\fix(F)\neq\varnothing$.
\end{theorem}

\begin{proof}
Since the proof is straightforward and differs only slightly from the reasoning contained in \cite[Theorem 3.2]{regan}, we only sketch it. Define \[\Sigma:=\bigcup\limits_{\lambdaup\in[0,1]}\f\{x\in\overline{U}^w\colon x\in(1-\lambdaup)x_0+\lambdaup F(x)\g\}.\] Since $E$ is weakly angelic and $\Sigma$ is sequentially closed (due to the regularity of $F$), the latter must be closed in the weak topology of $E$. Moreover, $\Sigma\subset U$, thanks to \eqref{L-S}. $(\overline{U}^w,w)$ is a normal subspace of the weakly normal space $E$. Let $\theta\colon(\overline{U}^w,w)\to[0,1]$ be an Urysohn mapping joining the closed disjoint sets $\Sigma$ and $\overline{U}^w\setminus U$ i.e., $\res{\theta}{\Sigma}\equiv 1$ and $\res{\theta}{\overline{U}^w\setminus U}\equiv 0$. Put $D:=\overline{\co}(\{x_0\}\cup F(\overline{U}^w))$. Now, define an auxiliary map $\hat{F}\colon D\map D$ by 
\[\hat{F}(x):=
\begin{cases}
(1-\theta(x))x_0+\theta(x)F(x)&\text{for }x\in D\cap U\\
\{x_0\}&\text{for }x\in D\setminus U.
\end{cases}\]
Clearly, $\hat{F}\colon(D,w)\map(D,w)$ is sequentially upper semicontinuous and has weakly compact convex values. Consider $M=\overline{\co}(\{x_0\}\cup\hat{F}(M))$. Observe that \[M\cap U\subset\overline{\co}(\{x_0\}\cup\hat{F}(M))=\overline{\co}(\{x_0\}\cup F(M\cap U)).\] By \eqref{monch}, $M\cap U$ is rwc. In view of the Krein-Smulian theorem $M$ is weakly compact. In other words, the map $\hat{F}$ satisfies the following M\"onch type condition 
\begin{equation}\label{monch2}
M\subset D,\;M=\overline{\co}(\{x_0\}\cup\hat{F}(M))\Rightarrow M\text{ is weakly compact.}
\end{equation}
Define \[{\mathcal A}:=\f\{M\in 2^{D}\setminus\{\varnothing\}\colon M\text{ is closed convex and } \overline{\co}\f(\{x_0\}\cup\hat{F}(M)\g)\subset M\g\}.\] Then the intersection $M_0:=\bigcap\limits_{M\in{\mathcal A}}M$ is nonempty (for $D\in{\mathcal A}$) and $M_0=\overline{\co}\f(\{x_0\}\cup\hat{F}(M_0)\g)$. By \eqref{monch2}, $M_0$ is weakly compact. Furthermore, it is $\hat{F}$-invariant. Observe that $(M_0,w)$ is a compact metrizable space due to $w^*$-separability of $E$. Thus, the set-valued map $\hat{F}\colon(M_0,w)\map(M_0,w)$ is compact admissible. Since $(M_0,w)$ is also an acyclic absolute extensor for the class of metrizable spaces (Dugundji's theorem), the map $\hat{F}$ possesses a fixed point by virtue od \cite[Theorem 7.4]{gorn2}. Obviously, it is also a fixed point for $F$.
\end{proof}

\begin{corollary}
Let $E$ be a separable Banach space. Assume that hypotheses $(\A_1)$, $(\F_1)$-$(\F_3)$ and $(\F_5)$ are satisfied. Assume also that instead of $(\F_4)$ the following growth condition holds
\begin{equation}\label{mu}
||F(t,x)||^+\<\mu(t)(1+|x|)\;\;\text{a.e. on }I\text{ for every }x\in E\text{ with }\mu\in L^1(I).
\end{equation}
Then the solution set $S_{\!F}(x_0)$ of the Cauchy problem \eqref{P} is nonempty.
\end{corollary}

\begin{proof}
Thanks to the exponential boundedness of the semigroup $\{S(t)\}_{t\geqslant 0}$ and the growth condition \eqref{mu} the set $\Sigma:=\{x\in C(I,E)\colon\exists\,\lambdaup\in[0,1]\;x\in\lambdaup H(x)\}$ is bounded (simply use the Gronwall inequality). Moreover, almost analogous reasoning to the one carried out in the proof of Theorem \ref{existence} (the paragraph in which we have shown that $M_0$ is weakly compact subset of $C(I,E)$) proves that $\Sigma$ is rwc. Fix a radius $M>0$ such that $\Sigma\subset D_C(0,M)$ and $\xi\in{\mathcal E}:=C(I,E)^*$ such that $||\xi||_{\mathcal E}\<1$. Then $\Sigma\subset \xi^{-1}([-M,M])$. Define $U:=\xi^{-1}((-M-1,M+1))$. Then $U$ is a weakly open neighbourhood of zero such that $\Sigma\cap\overline{U}^w\setminus U=\varnothing$. In particular Leray-Schauder boundary condition \eqref{L-S} is satisfied. In the proof of Theorem \ref{existence} we have also shown that \[M\subset C(I,E)\text{ bounded}, M\subset\overline{\co}(\{0\}\cup H(M))\Rightarrow M\text{ is rwc}.\] Hence, the M\"onch type condition \eqref{monch} is also satisfied. Since $H\colon(\overline{U}^w,w)\map(C(I,E),w)$ is convex weakly compact valued sequentially upper semicontinuous map and $C(I,E)$ is separable, $\fix(H)$ is nonempty in view of Theorem \ref{weak}.
\end{proof}

\par The eponymous topology of integrated solutions expresses itself in the following structure theorem, formulated in the context of a separable Banach space $E$.
\begin{theorem}\label{solset}
Let $E$ be a separable Banach space. Assume that $A\colon D(A)\to E$ satisfies $(\A_2)$ and $F\colon I\times E\map E$ satisfies $(\F_1)$-$(\F_5)$. Then the solution set of the Cauchy problem \eqref{P} is $R_\delta$.
\end{theorem}

\begin{remark}
For all we know, regarding our proof, the topological assumption about the separability of the space $E$ is indispensable. Application of \cite[Th.2.8.]{kunze} requires us to assume that $E$ is a weakly compactly generated Banach space. On the other hand the metrization theorem \cite[Prop.3.107]{fabian} holds for $E^*$ being $w^*$-separable. However, in view of \cite[Th.13.3]{fabian} a WCG space $E$ with $w^*$-separable dual $E^*$ must be compactly generated.
\end{remark}

\begin{remark}
The assumption regarding the equicontinuity of the semigroup $\{S(t)\}_{t\geqslant 0}$ is not in fact excessively restrictive. This is still a weaker requirement than assuming that $A$ satisfies the Hille-Yosida condition, which characterizes generators of locally Lipschitz continuous integrated semigroups.
\end{remark}

\begin{proof}
\par If the Banach space $E$ is separable, then the topological dual $E^*$ furnished with the $w^*$-topology $\sigma(E^*,E)$ is also separable. Suppose that $\x{x_n^*}$ is a countable $\sigma(E^*,E)$-dense subset of the unit sphere in $E^*$. Using this sequence, we are allowed to define a metric $d$ on $E$ in the following way:
\begin{equation}\label{d}
d(x,y):=\sum_{n=1}^\infty 2^{-n}|\langle x^*_n,x-y\rangle|.
\end{equation}
Clearly, this $d$-metric topology is weaker then the weak topology $\sigma(E,E^*)$ on $E$. Moreover, the $d$-metric topology and the weak topology coincide on the weakly compact subsets of $E$ (cf.\cite[Prop.3.107]{fabian}).\par We claim that there is a nonempty weakly compact convex set $X\subset C(I,E)$ possessing the following property:
\begin{equation}\label{X2}
S(t)x_0+\int_0^tS(t-s)\,\overline{\co}F(s,\overline{X(s)})\,ds\subset X(t)\;\;\text{for every }t\in I.
\end{equation}
Let $X_0=D_C(0,R)$ and $X_n=\overline{Y_n}$, where $R>0$ is such that $||S_{\!F}(x_0)||^+\<R$. Put \[Y_n=\left\{y\in C(I,E)\colon y(t)\in S(t)x_0+\int_0^tS(t-s)\,\overline{\co}F(s,\overline{X_{n-1}(s)})\,ds\mbox{ for }t\in I\right\}.\]
One easily sees that sets $X_n$ are well-defined nonempty bounded convex and equicontinuous (equicontinuity follows by \eqref{equicontinuity}, remaining properties are justified in \cite[Th.6.]{pietkun}). Moreover, $S_F(x_0)\subset X:=\bigcap\limits_{n=0}^\infty X_n$.\par Using the Castaing representation for the Hausdorff continuous multimap $t\mapsto\overline{X_n(t)}$, we may write $\overline{X_n(t)}=\overline{\{u_k(t)\}_{k=1}^\infty}$ with $\{u_k(t)\}_{k=1}^\infty\subset Y_n(t)$. Bearing in mind that $\lim\limits_{L\to\infty}\varphi(L)=0$, where $\varphi$ is a mapping given by \eqref{phi(L)}, we choose $L_0>0$ so that $Me^{\omega T}\varphi(L_0)<1$. Let $u_k=S(\cdot)x_0+V(f_k)$ for some $f_k\in N_{\overline{co}\,F}\f(\overline{X_{n-1}(\cdot)}\g)$. In view of \cite[Th.2.8.]{kunze}, we have
\begin{align*}
\sup_{t\in I}e^{-L_0t}\beta(X_n(t))&=\sup_{t\in I}e^{-L_0t}\beta\f(\{u_k(t)\}_{k=1}^\infty\g)=\sup_{t\in I}e^{-L_0t}\beta(\{S(t)x_0+V(f_k)(t)\})\\&\<\sup_{t\in I}e^{-L_0t}Me^{\omega t}\!\!\int_0^t\!\beta\f(\{f_k(s)\}_{k=1}^\infty\g)\,ds\<Me^{\omega t}\sup_{t\in I}e^{-L_0t}\!\!\int_0^t\!\eta(s)\beta(X_{n-1}(s))\,ds\\&\<Me^{\omega t}\sup_{t\in I}e^{-L_0t}\int_0^te^{L_0s}\eta(s)\,ds\sup_{t\in I}e^{-L_0t}\beta(X_{n-1}(t))\\&=Me^{\omega t}\varphi(L_0)\sup_{t\in I}e^{-L_0t}\beta(X_{n-1}(t)).
\end{align*}
Clearly, $\sup\limits_{t\in I}e^{-L_0t}\beta(X_n(t))\xrightarrow[n\to\infty]{}0$, which means that $\sup\limits_{t\in I}e^{-L_0t}\beta(X(t))=0$.\par Since $\beta(X(I))=\beta\f(\bigcup\limits_{t\in I}X(t)\g)=\sup\limits_{t\in I}\beta(X(t))$ (by \cite[Lem.2.]{szufla}), the topological subspace $\f(\overline{X(I)},\sigma(E,E^*)\g)$ is metrizable by $d$ (defined by \eqref{d}). We claim that $X$ is contained in a compact subspace of the space $C\f(I,\f(\overline{X(I)},d\g)\g)$ furnished with the topology of uniform convergence. Take $\eps>0$. There exists $n_0\in\mathbb{N}$ such that for all $n\geqslant n_0$ and every $t,\tau\in I$ we have \[\sup_{x\in X}\sum_{n=n_0}^\infty 2^{-n}|\langle x^*_n,x(t)-x(\tau)\rangle|<\eps/2.\] On the other hand, there exists $\delta_i>0$ for $i\in\{1,\ldots,n_0-1\}$ such that for every $t,\tau\in I$ with $|t-\tau|<\delta_i$ we have \[2^{-i}\sup_{x\in X}|\langle x_i^*,x(t)-x(\tau)\rangle|\<\sup_{x\in X}|\langle x_i^*,x(t)-x(\tau)\rangle|<\eps/2(n_0-1).\] Whence, \[\sup_{x\in X}d(x(t),x(\tau))\<\sup_{x\in X}\sum_{n=1}^{n_0-1}2^{-n}|\langle x_n^*,x(t)-x(\tau)\rangle|+\sup_{x\in X}\sum_{n=n_0}^{\infty}2^{-n}|\langle x_n^*,x(t)-x(\tau)\rangle|<\eps\] for every $t,\tau\in I$ such that $|t-\tau|<\delta=\min\limits_{1\<i\<n_0-1}\delta_i$. In other words, $X$ is equicontinuous with respect to $d$. At the same time, the cross-section $X(t)$ is relatively compact in $\f(\overline{X(I)},d\g)$ for every $t\in I$. Consequently, $X$ is relatively compact in $C\f(I,\f(\overline{X(I)},d\g)\g)$, by virtue of Ascoli's theorem.\par Observe that the inclusion mapping $i\colon C\f(I,\f(\overline{X(I)},d\g)\g)\hookrightarrow(C(I,E),w)$ is continuous. Therefore, $X$ is relatively weakly compact in $C(I,E)$. In fact, $X$ is weakly compact, since it is weakly closed. Consequently, $X(I)$ is weakly compact as well. Property \eqref{X2} easily follows from the fact that
\[\left\{y\in C(I,E)\colon y(t)\in S(t)x_0+\int_0^tS(t-s)\overline{\co}F(s,\overline{X(s)})\,ds\mbox{ for }t\in I\right\}\subset Y_n\] for every $\n$.\par For $A\subset E$, let $\dist(x,A):=\inf\limits_{y\in A}d(x,y)$. By $P\colon I\times E\map E$ we will denote the $d$-metric projection on the subset $\overline{X(t)}$, i.e. \[P(t,x):=\f\{y\in\overline{X(t)}\colon d(x,y)=\dist(x,X(t))\g\}.\] Since $\overline{X(t)}$ is $d$-compact, $P(t,x)$ must be nonempty. Relying on the weak compactness of the designed set $X$ we define an auxiliary multimap $\w{F}\colon I\times E\map E$ by the formula \[\w{F}(t,x):=\overline{\co}\,F(t,P(t,x)).\] Property \eqref{X2} plays a key role in proving that $S_{\w{F}}(x_0)=S_F(x_0)$. One can show that $\w{F}$ satisfies conditions $(\F_1)$-$(\F_5)$. Clearly, $\w{F}$ is integrably bounded and the map $\w{F}(t,\cdot)$ is weakly compact a.e. on $I$.\par Properties $(\F_2)$ and $(\F_3)$ require some commentary. Firstly, observe that the metric space $(E,d)$ is separable. Since $t\mapsto X(t)$ is Hausdorff continuous in the norm topology of $E$, $X(\cdot)\colon I\map(E,d)$ is measurable and $I\ni t\mapsto\dist(x,X(t))+\frac{1}{n}\in\R{}$ is continuous. Thus, $G_n\colon I\map(E,d)$ given by $G_n(t):=\f\{y\in E\colon d(x,y)\<\dist(x,X(t))+\frac{1}{n}\g\}$ is weakly measurable. Notice that \[P(t,x)=X(t)\cap\bigcap_{n=1}^\infty G_n(t).\] In view of \cite[Th.4.1.]{himmelberg}, the set-valued map $P(\cdot,x)\colon I\map(E,d)$ is measurable. Consequently, the codomain restriction $P(\cdot,x)\colon I\map(X(I),d)$ of $P(\cdot,x)$ constitutes a measurable multimap. Since $(X(I),d)$ is a Polish space, there exists a measurable $p_x\colon I\to(X(I),d)$ such that $p_x(t)\in P(t,x)$ for $t\in I$ (\cite[Th.5.1]{himmelberg}). Consider a sequence $(p_n\colon I\to X(I))_{n=1}^\infty$ of simple functions such that $d(p_n(t),p_x(t))\xrightarrow[n\to\infty]{}0$ a.e. on $I$, i.e. $p_n(t)\xrightharpoonup[n\to\infty]{E}p_x(t)$ a.e. on $I$. In accordance with $(\F_2)$, there exists a measurable $f_n\colon I\to(E,|\cdot|)$ such that $f_n(t)\in F(t,p_n(t))$ a.e. on $I$. In view of $(\F_4)$ the family $\x{f_n}$ is uniformly integrable. By $(\F_5)$ the cross-section $\x{f_n(t)}$ is relatively weakly compact in $E$. Therefore, $\z{f_n}$ is relatively weakly compact in $L^1(I,E)$, by Theorem \ref{ulger}. Assume that $f_n\xrightharpoonup[n\to\infty]{L^1(I,E)}f$, passing to a subsequence if necessary. According to Corollary \ref{wuhc}, $f(t)\in F(t,p_x(t))$ a.e. on $I$. Hence, $\w{F}(\cdot,x)$ has a measurable selection.\par Let $x_n\xrightharpoonup[n\to\infty]{E}x$. Fix $x^*\in E^*$. Obviously, there exists $z_n\in P(t,x_n)$ such that \[\sigma(x^*,F(t,P(t,x_n)))=\sigma(x^*,F(t,z_n))\;\;\text{ for }\n.\] From the very definition of $P$ follows that there is a subsequence $z_{k_n}\xrightharpoonup[n\to\infty]{E}z\in P(t,x)$. Thus, 
\begin{align*}
\overline{\lim_{n\to\infty}}\sigma(x^*,F(t,P(t,x_{k_n})))&=\overline{\lim_{n\to\infty}}\sigma(x^*,F(t,z_{k_n}))\<\sigma(x^*,F(t,z))\<\sigma(x^*,F(t,P(t,x)))\\&=\sigma(x^*,\w{F}(t,x))
\end{align*}
and eventually \[\overline{\lim_{n\to\infty}}\sigma(x^*,\w{F}(t,x_n))\<\sigma(x^*,\w{F}(t,x)).\] The latter means that $\w{F}$ satisfies $(\F_3)$.\par Let us define a set-valued approximation $F_n\colon I\times E\map E$ of the map $\w{F}$ in a routine manner, i.e. \[F_n(t,x):=\sum_{y\in E}\psi^n_y(x)\,\overline{\co}\,\w{F}(t,B_d(y,2r_n)),\] where $r_n:=3^{-n}$, $B_d(y,2r_n)$ is the ball considered in the metric space $(E,d)$ and the family $\f\{\psi^n_y\colon(E,d)\to[0,1]\g\}_{y\in E}$ is a locally Lipschitz partition of unity whose supports form a locally finite covering inscribed into the covering $\f\{B_d(y,r_n)\g\}_{y\in E}$ of the space $(E,d)$. Moreover, for every $\n$ define a mapping $f_n\colon I\times E\to E$ in the following way:\[f_n(t,x):=\sum_{y\in E}\psi^n_y(x)g_y(t)\in F_n(t,x),\] where $g_y$ is a measurable selection of $\w{F}(\cdot,y)$.\par If $H_n\colon C(I,E)\map C(I,E)$ is an operator given by $H_n:=S(\cdot)x_0+V\circ N_{F_n}$, then the topological space $(\fix(H_n),\sigma(C(I,E),C(I,E)^*))$ is $d$-compact metrizable. Indeed. Observe that $\varnothing\neq S_{\!F}(x_0)=S_{\!\w{F}}(x_0)\subset\fix(H_k)$, by Theorem \ref{existence}. and \eqref{X2}. Let $\z{u_n}\subset\fix(H_k)$. Then $u_n=S(\cdot)x_0+V(f_n)$, where $f_n\in N_{F_k}(u_n)$. Let's remind that $F_k(t,x)\subset\overline{\co}\w{F}(t,B(x,3r_k))$. Therefore,\[|f_n(t)|\<||F_k(t,u_n(t))||^+\<||\overline{\co}\w{F}(t,B(u_n(t),3r_k))||^+\<||F(t,X(t))||^+\] a.e. on $I$ and
\[\lim_{\ell(J)\to 0}\sup_\n\int\limits_J|f_n(t)|\,dt\<\lim_{\ell(J)\to 0}\overline{\int\limits_J}\sup_{|x|\<||X||^+}||F(t,x)||^+\,dt\<\lim_{\ell(J)\to 0}\int\limits_J g(t)\,dt=0,\] for some $g\in L^1(I)$.
On the other hand, \[\beta\f(\x{f_n(t)}\g)\<\beta\f(\overline{\co}\w{F}\f(t,B(\x{u_n(t)},3r_k)\g)\g)\<\beta(F(t,X(t)))\<\eta(t)\beta(X(t))\] for a.a. $t\in I$. In view of Theorem \ref{ulger}. the sequence $\z{f_n}$ is relatively weakly compact in $L^1(I,E)$. Hence we may assume, passing to a subsequence if necessary, that $f_n\xrightharpoonup[n\to\infty]{L^1(I,E)}f$. As a result, $u_n\xrightharpoonup[n\to\infty]{C(I,E)}u:=S(\cdot)+V(f)$. We would be done, if we only could demonstrate that $f\in N_{F_k}(u)$. Consider $x_n\xrightharpoonup[n\to\infty]{E}x$. Since $\overline{\co}\,\w{F}(t,B(y,2r_k))\subset\overline{\co}\,F(t,X(t))$ and $F(t,\cdot)$ is quasicompact in the weak topology, the map $F_k$ has weakly compact values by the Kre\v{\i}n-Smulian theorem. Hence, there exists $y_n\in F_k(t,x_n)$ such that $\sigma(x^*,F_k(t,x_n))=\langle x^*,y_n\rangle$ for $\n$ and some fixed $x^*\in E^*$. Since $\overline{\co}\,F(t,X(t))$ is also weakly compact, $y_n\xrightharpoonup[n\to\infty]{E}y$, up to a subsequence. Moreover, for every $m\geqslant 1$ there exists $z_m\in\co\{y_n\}_{n=m}^\infty$ such that $z_m\xrightarrow[m\to\infty]{E}y$. From the very definition of $F_k$ it follows that there exists $\gamma>0$ such that for all $x_1,x_2\in X(I)$ 
\begin{align*}
h(F_k(t,x_1),F_k(t,x_2))&\<\sum_{y\in E}|\psi^n_y(x_1)-\psi^n_y(x_2)|\,||\,\overline{\co}\,\w{F}(t,B_d(y,2r_k))||^+\\&\<\gamma\,||F(t,X(t))||^+\,d(x_1,x_2)
\end{align*} 
a.e. on $I$. Therefore, \[\lim_{n\to\infty}\inf_{z\in F_k(t,x)}|y_n-z|\<\lim_{n\to\infty}h(F_k(t,x_n),F_k(t,x))=0,\] i.e. \[\forall\,\eps>0\,\exists\,N\in\mathbb{N}\,\forall\,n\geqslant N\;\;\;y_n\in B(F_k(t,x),\eps).\] Whence, \[\forall\,\eps>0\,\exists\,N\in\mathbb{N}\,\forall\,m\geqslant N\;\;\;z_m\in B(F_k(t,x),\eps)\] and eventually $y\in D(F_k(t,x),\eps)$ for every $\eps>0$. This means that condition \eqref{wuhc2} in Corollary \ref{wuhc}. is met, since \[\limsup_{n\to\infty}\sigma(x^*,F_k(t,x_n))=\lim_{n\to\infty}\langle x^*,y_n\rangle=\langle x^*,y\rangle\<\sigma(x^*,F_k(t,x)).\] Now, since
\[\begin{cases}
u_n(t)\xrightharpoonup[n\to\infty]{E}u(t)&\text{for }t\in I\\
f_n\xrightharpoonup[n\to\infty]{L^1(I,E)}f&\\
f_n(t)\in F_k(t,u_n(t))&\text{a.e. on }I\\
\end{cases}\] 
we infer finally that $f(t)\in F_k(t,u(t))$ for a.a. $t\in I$. Therefore $\fix(H_k)$ is weakly compact and forms a $d$-compact metrizable subspace of the separable space $C(I,E)$.\par It is easy to see that $S_{\!\w{F}}(x_0)=\bigcap\limits_{n=1}^\infty\fix(H_n)$. The inclusion ''$\subset$'' is self-evident, since $\w{F}(t,x)\subset F_n(t,x)$. Let us take, then $u\in\bigcap\limits_{n=1}^\infty\fix(H_n)$. Suppose that $u=S(\cdot)x_0+V(f_n)$ with $f_n\in N_{F_n}(u)$. In analogous manner as previously we can prove that the sequence $\z{f_n}$ is relatively weakly compact in $L^1(I,E)$. Thus we may assume, passing to a subsequence if necessary, that $S(\cdot)x_0+V(f_n)\xrightharpoonup[n\to\infty]{C(I,E)}S(\cdot)x_0+V(f)$. Consider a subsequence $\z{f_{k_n}(t)}$ such that $f_{k_n}(t)\xrightharpoonup[n\to\infty]{E}z$. Fix $x^*\in E^*$. There exists $z_n\in\overline{P(t,B(u(t),3r_{k_n}))}^w$ such that $\sigma\f(x^*,F\f(t,\overline{P(t,B(u(t),3r_{k_n}))}^w\g)\g)=\sigma(x^*,F(t,z_n))$ for $\n$. Since $z_n=$w-$\lim\limits_{m\to\infty}y^n_m$ with $y_m^n\in P(t,B(u(t),3r_{k_n}))$, there is $w_m^n\in B(u(t),r_{k_n})$ such that $y_m^n\in P(t,w_m^n)$. The diagonalization procedure allows extraction of a subsequence $\z{y^n_{m_n}}$, which satisfies $d(z_n,y^n_{m_n})<\frac{1}{n}$. The strong convergence $w^n_{m_n}\xrightarrow[n\to\infty]{E}u(t)$ entails, passing to a subsequence if necessary, that $y^n_{m_n}\xrightharpoonup[n\to\infty]{E}y\in P(t,u(t))$. Hence, $z_n\xrightharpoonup[n\to\infty]{E}y\in P(t,u(t))$. Considering that $F(t,\cdot)$ is weakly sequentially uhc, we obtain
\begin{align*}
\langle x^*,z\rangle&=\lim_{n\to\infty}\langle x^*,f_{k_n}(t)\rangle\<\liminf_{n\to\infty}\sigma(x^*,\w{F}(t,B(u(t),3r_{k_n})))\\&=
\liminf_{n\to\infty}\sigma(x^*,F(t,P(t,B(u(t),3r_{k_n}))))\<\liminf_{n\to\infty}\sigma\f(x^*,F\f(t,\overline{P(t,B(u(t),3r_{k_n}))}^w\g)\g)\\&=\limsup_{n\to\infty}\sigma(x^*,F(t,z_n))\<\sigma(x^*,F(t,y))\<\sigma(x^*,F(t,P(t,u(t))))=\sigma(x^*,\w{F}(t,u(t))).
\end{align*}
Whence $z\in F(t,u(t))$, which means that $\overline{\co}\,$w-$\limsup\limits_{n\to\infty}\,\{f_n(t)\}\subset F(t,u(t))$. In view of \cite[Prop.2.3.31]{papa} it is clear that $f(t)\in F(t,u(t))$ a.e. on $I$. Thus $u=S(\cdot)x_0+V(f)\in S(\cdot)x_0+(V\circ N_F)(u)$, proving that $\bigcap\limits_{n=1}^\infty\fix(H_n)\subset S_{\!\w{F}}(x_0)$.\par Fix $\K$. Observe that for each $x\in\fix(H_k)$ there exists exactly one $f\in N_{F_k}(x)$ such that $x=S(\cdot)x_0+V(f)$. This follows directly from the fact that $x$ as an integrated solution has the form $x(t)=tx_0+A\int_0^tx(s)\,ds+\int_0^t(t-s)f(s)\,ds$ for $t\in I$. Since the univalent map $f_k\colon I\times E\to E$ satisfies conditions $(\F_1)$-$(\F_5)$, the integral equation 
\begin{equation}\label{equhomo}
u(t)=S(t)x_0+\int_0^\tau S(t-s)f(s)\,ds+\int_\tau^t S(t-s)f_k(s,u(s))\,ds\;\;\text{for }t\in[\tau,T],
\end{equation}
with $f\in N_F(x)$, possesses a solution (one can justify it easily analyzing carefully the proof of Theorem \ref{existence}.). \par It is worthwhile to notice that for any weakly compact subset $C\subset E$ there exists a constant $\gamma_C>0$ such that $|f_k(t,x_1)-f_k(t,x_2)|\<\gamma_C||F(t,X(t))||^+d(x_1,x_2)$ for all $x_1,x_2\in C$ and for a.a. $t\in I$. If $u_1,u_2$ are two solutions of equation \eqref{equhomo}, then 
\begin{align*}
|u_1(t)-u_2(t)|&\<\int_\tau^t||S(t-s)||_{{\mathcal L}}|f_k(s,u_1(s))-f_k(s,u_2(s))|\,ds\\&\<\int_\tau^t||S(t-s)||_{{\mathcal L}}\gamma_{u_1(I)\cup u_2(I)} g(s)\,d(u_1(s),u_2(s))\,ds\\&\<Me^{\omega T}\gamma_{u_1(I)\cup u_2(I)}\int_\tau^tg(s)|u_1(s)-u_2(s)|\,ds,
\end{align*}
where $g\in L^1(I)$ is such that $g(t)\geqslant\sup\limits_{|x|\<||X||^+}||F(t,x)||^+$ a.e. on $I$. Thus, equation \eqref{equhomo} has a unique solution $u[\tau,x(\tau)]$.\par Let $h\colon[0,1]\times\fix(H_k)\to\fix(H_k)$ be a homotopy given by the formula \[h(\lambdaup,x)(t):=
\begin{cases}
x(t),&t\in[0,\lambdaup T]\\
u[\lambdaup T;x(\lambdaup T)](t),&t\in[\lambdaup T,T].
\end{cases}\]
No need to emphasize that $h$ is well-defined. Observe that $h(0,x)=u_0$ for all $x\in\fix(H_k)$, where \[u_0(t)=S(t)x_0+\int_0^tS(t-s)f_k(s,u_0(s))\,ds,\;\;\;t\in I.\] At the same time $h(1,\cdot)=id_{\fix(H_k)}$. Assume that $\z{x_n}\subset\fix(H_k)$ and $\z{\lambdaup_n}\subset[0,1]$ are such that $x_n\xrightarrow[n\to\infty]{\fix(H_k)}x$ and $\lambdaup_n\xrightarrow[n\to\infty]{}\lambdaup$. For definiteness, let $\lambdaup_n\nearrow\lambdaup$. There are two cases to consider: $t<\lambdaup T$ and $t\geqslant\lambdaup T$. If $t<\lambdaup T$, then we are simply dealing with the convergence $h(\lambdaup_n,x_n)(t)\xrightharpoonup[n\to\infty]{E}x(t)=h(\lambdaup,x)(t)$. Suppose then, that $t\geqslant\lambdaup T$, $x_n=S(\cdot)x_0+V(f_n)$, $x=S(\cdot)x_0+V(f)$, $u_n:=h(\lambdaup_n,x_n)$, $u:=h(\lambdaup,x)$ and $\{x^*_m\}_{m=1}^\infty$ is a $w^*$-dense subset of the dual $E^*$. From Theorem \ref{ulger}. it follows easily that $f_n\xrightharpoonup[n\to\infty]{L^1(I,E)}g$, up to a subsequence. Thus, $x_n\xrightharpoonup[n\to\infty]{C(I,E)}S(\cdot)x_0+V(g)$ and consequently $V(f)=V(g)$. Eventually $f_n\xrightharpoonup[n\to\infty]{L^1(I,E)}f$, by Lemma \ref{mono}. Observe that \[\int_0^{\lambdaup T}S(t-s)f_n(s)\,ds\xrightharpoonup[n\to\infty]{E}\int_0^{\lambdaup T}S(t-s)f(s)\,ds\] and \[\int_{\lambdaup_n T}^{\lambdaup T}S(t-s)f_n(s)\,ds\xrightarrow[n\to\infty]{E}0.\] Therefore 
\begin{equation}\label{weakcon}
\int_0^{\lambdaup_n T}S(t-s)f_n(s)\,ds\xrightharpoonup[n\to\infty]{E}\int_0^{\lambdaup T}S(t-s)f(s)\,ds.
\end{equation}
\par Notice that $K:=\overline{\bigcup_{n=1}^\infty u_n(I)}^w\cup u(I)$ is weakly compact. Estimates on the segment $[\lambdaup T,T]$ have the following form: 
\begin{align*}
\beta\f(\x{u_n(t)}\g)&=\beta\f(\f\{S(t)x_0+\int_0^{\lambdaup_n T}S(t-s)f_n(s)\,ds+\int_{\lambdaup_n T}^t S(t-s)f_k(s,u_n(s))\,ds\g\}_{n=1}^\infty\g)\\&\<Me^{\omega T}\int_0^{\lambdaup_n T}\beta\f(\x{f_n(s)}\g)\,ds+Me^{\omega T}\int_{\lambdaup_n T}^t\beta\f(f_k\f(s,\x{u_n(s)}\g)\g)\,ds\\&\<Me^{\omega T}\int_0^t\eta(s)\beta(X(s))\,ds=0
\end{align*}
and 
\begin{align*}
\sup_\n|u_n(t)-u_n(\tau)|&\<|S(t)x_0-S(\tau)x_0|+\sup_\n\int_0^{\lambdaup_n T}||S(t-s)-S(\tau-s)||_{\mathcal L}|f_n(s)|\,ds\\&+\sup_\n\int_{\lambdaup_n T}^\tau||S(t-s)-S(\tau-s)||_{\mathcal L}|f_k(s,u_n(s))|\,ds\\&+\sup_\n\int_\tau^t||S(t-s)||_{\mathcal L}|f_k(s,u_n(s))|\,ds\\&\<|S(t)x_0-S(\tau)x_0|+\int_0^\tau||S(t-s)-S(\tau-s)||_{\mathcal L}g(s)\,ds\\&+Me^{\omega T}\int_\tau^tg(s)\,ds,
\end{align*}
where $g\in L^1(I)$ satisfies $g(t)\geqslant\sup\limits_{|x|\<||X||^+}||F(t,x)||^+$ a.e. on $I$. Therefore, $\bigcup\limits_{n=1}^\infty u_n(I)$ is strongly equicontinuous and $\beta\f(\bigcup\limits_{n=1}^\infty u_n(I)\g)=\sup\limits_{t\in I}\beta\f(\x{u_n(t)}\g)=0$. Whence weak compactness of $K$ follows. In that connection, for every $\n$ and for a.a. $t\in I$ one has
\begin{equation}\label{estimate}
|f_k(t,u_n(t))-f_k(t,u(t))|\<\gamma_K||F(t,X(t))||^+d(u_n(t),u(t)).
\end{equation}
\par For every $\eps>0$ there exists $m_0\in\mathbb{N}$ such that \[\sum_{m=m_0}^\infty2^{-m}|\langle x^*_m,u_n(t)-u(t)\rangle|\<\frac{\eps}{3}.\] 
By virtue of \eqref{weakcon} and \eqref{estimate} we may choose an $N\in\mathbb{N}$ such that for $m\in\{1,\ldots,m_0-1\}$ and $n\geqslant N$ we obtain
\begin{align*}
|\langle x^*_m,u_n(t)-u(t)\rangle|&\<\f|\f\langle x^*_m,\int_0^{\lambdaup_n T}S(t-s)f_n(s)\,ds-\int_0^{\lambdaup T}S(t-s)f(s)\,ds\g\rangle\g|\\&+\f|\f\langle x^*_m,\int_{\lambdaup_n T}^tS(t-s)f_k(s,u_n(s))\,ds-\int_{\lambdaup T}^tS(t-s)f_k(s,u(s))\,ds\g\rangle\g|\\&\<\frac{\eps}{3}+|x^*_m|\int_{\lambdaup_n T}^{\lambdaup T}||S(t-s)||_{{\mathcal L}}|f_k(s,u_n(s))|\,ds\\&+|x^*_m|\int_{\lambdaup T}^t||S(t-s)||_{{\mathcal L}}|f_k(s,u_n(s))-f_k(s,u(s))|\,ds\\&\<\frac{\eps}{3}+Me^{\omega T}\f(\int_{\lambdaup_n T}^{\lambdaup T}g(s)\,ds+\gamma_K\int_{\lambdaup T}^tg(s)\,d(u_n(s),u(s))\,ds\g)\\&\<\frac{2}{3}\eps+Me^{\omega T}\gamma_K\int_{\lambdaup T}^tg(s)\,d(u_n(s),u(s))\,ds,
\end{align*}
where $g\in L^1(I)$ satisfies $g(t)\geqslant\sup\limits_{|x|\<||X||^+}||F(t,x)||^+$ a.e. on $I$. Hence
\begin{align*}
d(u_n(t),u(t))&=\sum_{m=1}^\infty2^{-m}|\langle x^*_m,u_n(t)-u(t)\rangle|\<\frac{\eps}{3}+\sum_{m=1}^{m_0-1}2^{-m}|\langle x^*_m,u_n(t)-u(t)\rangle|\\&\<\frac{\eps}{3}+\sum_{m=1}^{m_0-1}2^{-m}\f(\frac{2}{3}\eps+Me^{\omega T}\gamma_K\int_{\lambdaup T}^tg(s)\,d(u_n(s),u(s))\,ds\g)\\&\<\eps+Me^{\omega T}\gamma_K\int_{\lambdaup T}^tg(s)\,d(u_n(s),u(s))\,ds
\end{align*}
for $n\geqslant N$. Eventually, \[d(u_n(t),u(t))\<\eps\exp(Me^{\omega T}\gamma_K||g||_1),\] i.e. $d(u_n(t),u(t))\xrightarrow[n\to\infty]{}0$ for $t\geqslant\lambdaup T$. In fact, we have shown that $h(\lambdaup_n,x_n)(t)\xrightharpoonup[n\to\infty]{E}h(\lambdaup,x)(t)$ for $t\in I$. Since \[\sup_{\stackrel{\scriptstyle \n}{\scriptstyle t\in I}}|h(\lambdaup_n,x_n)(t)|\<Me^{\omega T}(|x_0|+||g||_1)\] with $g\in L^1(I)$ such that $g(t)\geqslant\sup\limits_{|x|\<||X||^+}||F(t,x)||^+$ for a.a. $t\in I$, the latter entails weak convergence $h(\lambdaup_n,x_n)\xrightharpoonup[n\to\infty]{C(I,E)}h(\lambdaup,x)$. This means that $h$ is a continuous mapping with respect to the relative weak topology of $\fix(H_k)$.\par Summing up, the solution set $S_{\!F}(x_0)$ is representable in the form of the intersection of a decreasing sequence of compact contractible metric spaces $(\fix(H_n),d)$.
\end{proof}

\begin{corollary}\label{solsetcor}
Suppose that \eqref{mu} holds. Under assumptions of Theorem \ref{solset}, with the exclusion of condition $(\F_4)$, the solution set of the Cauchy problem \eqref{P} is $R_\delta$.
\end{corollary}

\section{Applications}

Formulated in the previous section Theorem \ref{solset}., on the geometric structure of the solution set $S_{\!F}(x_0)$, will allow us employ an approach imitating method of the operator of translation along the trajectories to demonstrate the existence of integrated solutions to the nonlocal Cauchy problem. Consider, therefore, the following boundary value problem:
\begin{equation}\label{N}
\begin{dcases*}
\dot{x}(t)\in Ax(t)+F\f(t,\int_0^tx(s)\,ds\g)&on $I$,\\
x(0)\in G(x),
\end{dcases*}
\end{equation}
where $G\colon C(I,E)\map E$. By applying mentioned approach, we were able to prove
\begin{theorem}\label{nonlocal}
Let $E$ be a separable Banach space. Assume that $A\colon D(A)\to E$ is a generator of a non-degenerate equicontinuous integrated semigroup $\{S(t)\}_{t\geqslant 0}$ such that $||S(t)||_{\mathcal L}\<e^{\omega t}$ for $t\geqslant 0$. Assume further that $F\colon I\times E\map E$ satisfies $(\F_1)$-$(\F_5)$. Let $G\colon C(I,E)\map E$ be a set-valued operator whose restriction $\res{G}{M}\colon(M,w)\map(E,w)$ to any weakly compact subset $M\subset C(I,E)$ is an admissible map. If $G$ satisfies 
\begin{itemize}
\item[$(\text{G}_1)$] $\beta(G(\Omega))\<\beta(\Omega(T))\;$ for bounded $\Omega\subset C(I,E)$,
\item[$(\text{G}_2)$] $\exists\,a,d>0\,\exists\,R>0\,\forall\,|x|\geqslant R\;\;\;||G(x)||^+\<a||x||^\alpha+d$ for some $\alpha\in(0,1)$
\end{itemize}
and 
\begin{equation}\label{omega}
\omega T+||\eta||_1<1,
\end{equation}
then the nonlocal Cauchy problem \eqref{N} has an integrated solution.
\end{theorem}

\begin{proof}
Define $P\colon E\map E$ by $P=G\circ S_F$, where $S_F\colon E\map C(I,E)$ is the solution set map given by \[S_{\!F}(x_0):=\f\{x\in C(I,E)\colon x\text{ is an integrated solution of }\eqref{P}\g\}.\] \par We will show that there exists $R>0$ such that $P(D(0,R))\subset D(0,R)$. Suppose not. Then there exist elements $x_n\in E$ and $y_n\in P(x_n)$ such that $|x_n|\<n$ and $|y_n|>n$ for $\n$. If $y_n\in G(u_n)$, then either $\z{u_n}$ is bounded or $||u_n||\xrightarrow[n\to\infty]{}+\infty$. In the first case there must be a radius $\hat{R}>0$ such that $G(\x{u_n})\subset D(0,\hat{R})$. Thus, \[1\<\limsup_{n\to\infty}\frac{|y_n|}{n}\<\limsup_{n\to\infty}\frac{||G(u_n)||^+}{n}\<\lim_{n\to\infty}\frac{\hat{R}}{n}=0\] - a contradiction. Assume that $||u_n||\xrightarrow[n\to\infty]{}+\infty$ and $\lim\limits_{n\to\infty}\frac{|x_n|}{||u_n||}=0$. Then \[1=\limsup_{n\to\infty}\frac{||u_n||}{||u_n||}\<e^{\omega T}\f(\limsup_{n\to\infty}\frac{|x_n|}{||u_n||}+\limsup_{n\to\infty}\frac{1}{||u_n||}\overline{\int\limits_I}\sup_{|x|\<||u_n||}||F(t,x)||^+\,dt\g)<1\] - a contradiction. Suppose than that $||u_n||\xrightarrow[n\to\infty]{}+\infty$ and that there exist $\eps>0$ and $\z{k_n}$ such that $|x_{k_n}|>\eps||u_{k_n}||$ for $\n$. Using hypothesis $(\text{G}_2)$ we get 
\begin{align*}
1\<\limsup_{n\to\infty}\frac{|y_{k_n}|}{k_n}&<\limsup_{n\to\infty}\frac{||G(u_{k_n})||^+}{k_n}\<\limsup_{n\to\infty}\frac{a||u_{k_n}||^\alpha+d}{k_n}\\&\<\limsup_{n\to\infty}\frac{a\eps^{-\alpha}(k_n)^\alpha+d}{k_n}=0.
\end{align*}
In other words, there must be a ball $D(0,R)$ invariant under the Poincar\'e-type operator $P$.\par Consider a sequence $x_n\xrightharpoonup[n\to\infty]{E}x$. Let $u_n\in S_{\!F}(x_n)$ be such that $u_n=S(\cdot)x_n+V(f_n)$. Clearly, the sequence of solutions $\z{u_n}$ possesses a priori bounds. Hence \[|f_n(t)|\<||F(t,u_n(t))||^+\<b(t)(1+|u_n(t)|)\<b(t)(1+\sup_\n||u_n||)\] for each $\n$ and for a.a. $t\in I$. At the same time
\begin{align*}
\beta(\x{u_n(t)})&=\beta\f(\x{S(t)x_n+V(f_n)(t)}\g)\\&\<||S(t)||_{\mathcal L}\beta(\x{x_n})+\int_0^t||S(t-s)||_{\mathcal L}\beta(\x{f_n(s)})\,ds\\&\<\int_0^t\eta(s)\beta(\x{u_n(s)})\,ds. 
\end{align*}
Hence, $\sup\limits_{t\in I}\beta(\x{u_n(t)})=0$ and eventually $\beta(\x{f_n(t)})=0$ a.e on $I$. By virtue of Theorem \ref{ulger}. one may assume, passing to a subsequence if necessary, that $f_n\xrightharpoonup[n\to\infty]{L^1(I,E)}f$. Obviously, $S(t)x_n\xrightharpoonup[n\to\infty]{E}S(t)x$. Consequently, $u_n(t)=S(t)x_n+V(f_n)(t)\xrightharpoonup[n\to\infty]{E}S(t)x+V(f)(t)=:u(t)$ for each $t\in I$. Since $\sup\limits_{\n}||u_n||<\infty$, the latter means that $u_n\xrightharpoonup[n\to\infty]{C(I,E)}u$. Since hypotheses of the Convergence Theorem are met (cf. Corollary \ref{wuhc}.), we infer that $f\in N_F(u)$. On that account $u\in S_{\!F}(x)$. So the operator $S_{\!F}\colon(M,w)\map(C(I,E),w)$ is a weakly compact valued upper semicontinuous map for each fixed relatively weakly compact subset $M\subset E$. Now we can apply the structure theorem (Theorem \ref{solset}.) to get admissibility of the Poinar\'e-type operator $P\colon(M,w)\map(E,w)$ (remember that the composition of two admissible maps is still admissible).\par Let us reiterate the reasoning contained in the proof of Theorem \ref{existence}. Fix $\hat{x}\in D(0,R)$ and define \[{\mathcal A}:=\f\{M\in 2^{D(0,R)}\setminus\{\varnothing\}\colon M\text{ is closed convex and } \overline{\co}\f(\{\hat{x}\}\cup P(M)\g)\subset M\g\}.\] Then the intersection $M_0:=\bigcap\limits_{M\in{\mathcal A}}M$ is nonempty and possesses the following form $M_0=\overline{\co}\f(\{\hat{x}\}\cup P(M_0)\g)$. We will show that $M_0$ is weakly compact in $E$. Let $u_n=S(\cdot)x_n+V(f_n)$ with $f_n\in N_F(u_n)$ and $x_n\in M_0$. Put \[\Delta(\Omega):=\f\{D\in 2^\Omega\setminus\{\varnothing\}\colon D\text{ is countable}\g\}.\] In view of \cite[Th.2.8.]{kunze} we have 
\begin{align*}
\beta\f(\x{u_n(t)}\g)&=\beta\f(\x{S(t)x_n+V(f_n)(t)}\g)\<\beta\f(S(t)\x{x_n}\g)+\beta\f(\f\{\int_0^tS(t-s)f_n(s)\,ds\g\}\g)\\&\<||S(t)||_{\mathcal L}\beta\f(\x{x_n}\g)+\int_0^t||S(t-s)||_{\mathcal L}\beta\f(\x{f_n(s)}\g)\,ds\\&\<e^{\omega t}\max_{D\in\Delta(M_0)}\beta(D)+\int_0^te^{\omega(t-s)}\eta(s)\beta\f(\x{u_n(s)}\g)\,ds.
\end{align*}
Defining the right-hand side of the above inequality by $\rho$, we see that
\[\rho'(t)=\omega\rho(t)+\eta(t)\beta\f(\x{u_n(t)}\g)\<(\omega+\eta(t))\rho(t)\] a.e. on $I$. Solving of this differential inequality leads to
\[\beta\f(\x{u_n(t)}\g)\<\rho(t)\<\max_{D\in\Delta(M_0)}\beta(D)\exp\f(\omega t+\int_0^t\eta(s)\,ds\g)\] for $t\in I$. Using the latter and $(\text{G}_1)$ we obtain
\begin{multline*}
\max_{D\in\Delta(M_0)}\beta(D)=\max_{D\in\Delta(P(M_0))}\beta(D)\<\max_{D\in\Delta(S_{\!F}(M_0))}\beta(G(D))\\\<\max_{D\in\Delta(S_{\!F}(M_0))}\beta(D(T))\<\max_{D\in\Delta(M_0)}\beta(D)\,e^{\omega T+||\eta||_1}.
\end{multline*}
In view of \eqref{omega}, $\max\limits_{D\in\Delta(M_0)}\beta(D)=0$. In accordance with the Eberlein-\v{S}mulyan theorem the set $M_0$ must be weakly compact.\par Summing up, the admissible operator $P\colon M_0\map M_0$ from the convex subset $M_0$ of the locally convex space $(E,w)$ to the compact metrizable subset of $M_0$ has at least one fixed point, by virtue of Theorem \ref{Lefschetz}. This fixed point constitutes a solution to the boundary value problem \eqref{N}.
\end{proof}

\begin{corollary}\label{nonlocal2}
Assume that an operator $G\colon C(I,E)\map E$ has a weakly sequentially closed graph, acyclic values, maps bounded sets into relatively weakly compact sets and satisfies the sublinear growth condition $(\text{G}_2)$. Then, taking into account the remaining hypotheses of Theorem \ref{nonlocal}. $($except of course the condition $(\text{G}_1))$, the nonlocal Cauchy problem \eqref{N} has a solution.
\end{corollary}

\begin{remark}
Corollary \ref{nonlocal2}. emphasizes the advantage of Theorem \ref{nonlocal}. over \cite[Th.2.2.]{vath}, at least in the context of a separable Banach space and non-degenerate integrated semigroups. It would not be possible to weaken the assumption regarding the topology of values of the boundary condition operator without using the structure theorem, i.e. Theorem \ref{solset}. 
\end{remark}

\par Now we turn our attention to the following multivalued wave equation:
\begin{equation}\label{wave}
\begin{dcases*}
\square\,u(t,x)=f_2(t,x)+\Delta\int_0^tf_1(s,x)\,ds\\
f_1(t,x)\in\f[h^1_1\f(t,x,\int_{\R{n}}k_1^1(t,y)\,dy\g),h^1_2\f(t,x,\int_{\R{n}}k_2^1(t,y)\,dy\g)\g]&\\
f_2(t,x)\in\f[h^2_1\f(t,x,\int_{\R{n}}k_1^2(t,y)u(t,y)\,dy\g),h^2_2\f(t,x,\int_{\R{n}}k_2^2(t,y)u(t,y)\,dy\g)\g]&\\
\end{dcases*}
\end{equation}
on $I\times\R{n}$, subject to the Cauchy condition
\begin{equation}\label{boundary}
\begin{dcases*}
\partial_tu(0,x)=\mathring{u}_2&on $\R{n}$\\
u(0,x)=\mathring{u}_1&on $\R{n}$,
\end{dcases*}
\end{equation}
where $\square$ is the d'Alembertian and $k_i^1(t,\cdot)\in L^1(\R{n})$, $k_i^2(t,\cdot)\in L^2(\R{n})$ for a.a. $t\in I$ and $i=1,2$.
\par Let $\langle\cdot,\cdot\rangle$ denote the inner product in $L^2(\R{n})$. By the weak solution of the problem \eqref{wave}-\eqref{boundary} we mean a function $w\in C(I,L^2(\R{n}))$ such that for every $v\in H^2(\R{n})$ the function $\langle w(\cdot),v\rangle\in W^{2,1}(I)$ and
\begin{equation*}
\begin{dcases*}
\frac{d^2}{dt^2}\,\langle w(t),v\rangle=\langle w(t),\Delta v\rangle+\langle f_2(t),v\rangle+\f\langle\int_0^tf_1(s)\,ds,\Delta v\g\rangle&a.e. on $I$\\
\res{\frac{d}{dt}\,\langle w(t),v\rangle}{t=0}=\langle\mathring{u}_2,v\rangle\\
w(0)=\mathring{u}_1
\end{dcases*}
\end{equation*}
for some functions $f_1,f_2\in L^1(I,L^2(\R{n}))$ such that
\begin{equation*}
\begin{dcases}
h_1^1\f(t,x,\int_{\R{n}}k_1^1(t,y)\,dy\g)\<f_1(t,x)\<h_2^1\f(t,x,\int_{\R{n}}k_2^1(t,y)\,dy\g)\\
h_1^2\f(t,x,\int_{\R{n}}\!k_1(t,y)(w(t,y)-\mathring{u}_1(y))\,dy\g)\<f_2(t,x)\<h_2^2\f(t,x,\int_{\R{n}}\!k_2(t,y)(w(t,y)-\mathring{u}_1(y))\,dy\g)
\end{dcases}
\end{equation*}
for a.a. $t\in I$ and a.a. $x\in\R{n}$. Let ${\mathcal S}(\mathring{u}_1,\mathring{u}_2)$ denote the set of all weak solutions of the problem \eqref{wave}-\eqref{boundary}. 
\par Our hypotheses on $h_j^i\colon I\times\R{n}\times\R{}\to\R{}$ are the following:
\begin{itemize}
\item[$(\text{h}_1)$] for any $u\in L^2(\R{n})$ there exist $v,w\in L^1(I,L^2(\R{n}))$ such that 
\begin{equation*}
\begin{dcases}
h^1_1\f(t,x,\int_{\R{n}}k_1^1(t,y)\,dy\g)\<v(t,x)\<h^1_2\f(t,x,\int_{\R{n}}k_2^1(t,y)\,dy\g)&\\
h^2_1\f(t,x,\int_{\R{n}}k_1^2(t,y)u(y)\,dy\g)\<w(t,x)\<h^2_2\f(t,x,\int_{\R{n}}k_2^2(t,y)u(y)\,dy\g)&\\
\end{dcases}
\end{equation*} 
for a.a. $t\in I$ and a.a. $x\in\R{n}$,
\item[$(\text{h}_2)$] for a.a. $t\in I$, for a.a. $x\in\R{n}$ and for every $z\in\R{}$ the functions $h^2_1(t,x,\cdot)$ is lower semicontinuous while $h^2_2(t,x,\cdot)$ is upper semicontinuous,
\item[$(\text{h}_3)$] for $j=1,2$ there exist $b_1,b_2\in L^1(I)$ and $c_1\colon I\times\R{n}\to\R{}$, $c_2\colon I\times\R{n}\times\R{}_+\to\R{}$ such that 
\begin{equation*}
\begin{dcases}
\sup_{|z|\<||k_j^1(t,\cdot)||_1}|h_j^1(t,x,z)|\<c_1(t,x)&\\
\sup_{|z|\<||k_j^2(t,\cdot)||_2r}|h_j^2(t,x,z)|\<c_2(t,x,r)&\\
\end{dcases}
\end{equation*} 
and 
\begin{equation*}
\begin{dcases}
\int_{\R{n}}c_1^2(t,x)\,dx\<b_1^2(t)&\\
\int_{\R{n}}c_2^2(t,x,r)\,dx\<b_2^2(t)(1+r)^2
\end{dcases}
\end{equation*} 
for every $r>0$, for a.a. $t\in I$ and for a.a. $x\in\R{n}$.
\end{itemize}

\begin{theorem}
If hypotheses $(\text{h}_1)$-$(\text{h}_3)$ hold, then for every $\mathring{u}_1,\mathring{u}_2\in L^2(\R{n})$ the solution set ${\mathcal S}(\mathring{u}_1,\mathring{u}_2)$ is acyclic in the space $C(I,L^2(\R{n}))$ endowed with the weak topology.
\end{theorem}
\begin{proof}
Let \[E:=L^2(\R{n})\times L^2(\R{n}),\] \[D(A):=H^2(\R{n})\times L^2(\R{n}),\] \[E_0:=H^1(\R{n})\times L^2(\R{n}),\] \[D(A_0):=H^2(\R{n})\times H^1(\R{n}).\] Assume that the Hilbert space $E$ is furnished with the norm \[||(x,y)||_E:=\f(||x||_2^2+||y||_2^2\g)^\frac{1}{2}\] and $E_0$ with \[|||(x,y)|||:=\f(||x||_2^2+\langle \nabla x,\nabla x\rangle_{L^2(\R{n},\R{n})}+||y||_2^2\g)^\frac{1}{2}.\] The linear operator $A\colon D(A)\to E$, given by $A(u_1,u_2):=(u_2,\Delta u_1)$, generates an exponentially bounded non-degenerate integrated semigroup $\{S(t)\}_{t\geqslant 0}$ on $E$ such that \[S(t)(\mathring{u}_1,\mathring{u}_2)=\f(\int_0^tw(s)\,ds,w(t)-\mathring{u}_1\g),\] where $w\in C^2([0,\infty),L^2(\R{n}))$ satisfies
\begin{equation*}
\begin{dcases}
\frac{d^2}{dt^2}\,\langle w(t),v\rangle=\langle w(t),\Delta v\rangle\\
\res{\frac{d}{dt}\,\langle w(t),v\rangle}{t=0}=\langle\mathring{u}_2,v\rangle\\
w(0)=\mathring{u}_1
\end{dcases}
\end{equation*}
for every $v\in H^2(\R{n})$ (see \cite[Th.7.1.]{thieme}). It is a consequence of the fact that the part $A_0\colon D(A_0)\to E_0$ of $A$ generates a strongly continuous semigroup $\{T_0(t)\}_{t\geqslant 0}$ on $(E_0,|||\cdot|||)$, satisfying $||T_0(t)||_{\mathcal L}\<e^{2t}$ (\cite[7.4.5.]{pazy}).\par We claim that the resolvent set $\rho(A)$ contains $(2,\infty)$. For every $(f_1,f_2)\in C^\infty_0(\R{n})\times C^\infty_0(\R{n})$ there exists a unique $(u_1,u_2)\in D(A)$ such that 
\[\begin{cases}
u_1-\lambdaup u_1=f_1,\\
u_2-\lambdaup\Delta u_1=f_2,
\end{cases}\]
for every real $\lambdaup\neq 0$ (see \cite[Lem.7.4.3.]{pazy}). Considering this, we are able to estimate:
\begin{align*}
||(f_1,f_2)||_E^2&=\langle u_1-\lambdaup u_2,u_1-\lambdaup u_2\rangle+\langle u_2-\lambdaup\Delta u_1,u_2-\lambdaup\Delta u_1\rangle\\&=||u_1||_2^2+\lambdaup^2||u_2||_2^2-\lambdaup\langle u_1,u_2\rangle+||u_2||_2^2-\lambdaup\langle\Delta u_1,u_2\rangle-\lambdaup\langle u_2,\Delta u_1\rangle+\lambdaup^2||\Delta u_1||_2^2\\&\geqslant||u_1||_2^2+||u_2||_2^2-2\lambdaup\langle u_1,u_2\rangle=||(u_1,u_2)||_E^2-2\lambdaup\langle u_1,u_2\rangle\geqslant(1-\lambdaup)||(u_1,u_2)||_E^2\\&\geqslant(1-2\lambdaup)^2||(u_1,u_2)||_E^2
\end{align*}
for $\lambdaup\in\f(0,\frac{1}{2}\g)$. In other words, for every $\lambdaup\in\f(0,\frac{1}{2}\g)$ and $f\in C^\infty_0(\R{n})\times C^\infty_0(\R{n})$ there exists a unique $u\in D(A)$ such that $u-\lambdaup Au=f$ and 
\begin{equation}\label{resolvent}
||u||_E\<(1-2\lambdaup)^{-1}||f||_E.
\end{equation}
Since $C^\infty_0(\R{n})\times C^\infty_0(\R{n})$ is dense in $E$ and the operator $A$ is closed, \eqref{resolvent} entails $\text{Im}(\lambdaup+A)=E$ for $\lambdaup>2$, i.e. $(2,\infty)\subset\rho(A)$. From \eqref{resolvent} it follows also that 
\begin{equation}\label{res}
||R_\lambdaup||_{\mathcal L}\<\frac{1}{\lambdaup-2}\;\;\text{for }\lambdaup>2,
\end{equation}
with $R_\lambdaup:=(\lambdaup+A)^{-1}$.\par  For every $(u_1,u_2)\in D(A_0)$ the norm $|||\cdot|||$ possesses the following bound:
\begin{align*}
|||(u_1,u_2)|||^2&=||u_1||_2^2+\int_{\R{n}}\langle\nabla u_1(x),\nabla u_1(x)\rangle_{\R{n}}\,dx+||u_2||_2^2=||u_1||_2^2-\langle\Delta u_1,u_1\rangle+||u_2||_2^2\\&\<||u_1||_2^2+||\Delta u_1||_2||u_1||_2+||u_2||_2^2\<||u_1||_2^2+\frac{1}{2}||u_1||_2^2+\frac{1}{2}||\Delta u_1||_2^2+||u_2||_2^2\\&\<\frac{3}{2}\f(||u_1||_2^2+||u_2||_2^2\g)+\frac{1}{2}\f(||u_2||_2^2+||\Delta u_1||_2^2\g)=\frac{3}{2}||(u_1,u_2)||_E^2+\frac{1}{2}||A(u_1,u_2)||_E^2.
\end{align*}
Whence \[||(u_1,u_2)||_E\<|||(u_1,u_2)|||\<\sqrt{2}\f(||(u_1,u_2)||_E+||A(u_1,u_2)||_E\g).\] That being said, for every initial value $x\in D(A_0)$ there exists a unique solution $u\in C^1(\R{}_+,D(A_0))$ of the abstract Cauchy problem
\begin{equation}\label{ACP}
\begin{cases}
\dot{u}(t)=Au(t),\\
u(0)=x,
\end{cases}
\end{equation}
satisfying $||u(t)||_E\<|||u(t)|||\<e^{2t}|||x|||\<\sqrt{2}e^{2t}(||x||_E+||Ax||_E)$. For $\lambdaup>2$, the function $w(t):=R_\lambdaup u(t)$ is a solution of \eqref{ACP} with  
\begin{align*}
||w(t)||_E&\<\sqrt{2}e^{2t}\f(||R_\lambdaup x||_E+||AR_\lambdaup x||_E\g)\<\sqrt{2}e^{2t}\f(||R_\lambdaup x||_E+\lambdaup||R_\lambdaup x||_E+||x||_E\g)\\&\<\sqrt{2}\f(\frac{1+\lambdaup}{\lambdaup -2}+1\g)e^{2t}||x||_E,
\end{align*}
by \eqref{res}. Let $v(t):=\int_0^tu(s)\,ds$ be the integrated solution. Then $v(t)=\lambdaup\int_0^tw(s)\,ds-w(t)+R_\lambdaup x$. Moreover, the operator $A$ generates an integrated semigroup $\{S(t)\}_{t\geqslant 0}$, given by $S(t)x=v(t)$ for $x$ taken from the dense subspace $D(A_0)$ of the space $E$ (see \cite[Th.4.2.]{neu}). Therefore,
\begin{align*}
||S(t)x||_E&\<\lambdaup\int_0^t||w(s)||_E\,ds+||w(t)||_E+||R_\lambdaup x||_E\\&\<(1+\lambdaup)\sqrt{2}\f(\frac{1+\lambdaup}{\lambdaup -2}+1\g)e^{2t}||x||_E+\frac{1}{\lambdaup-2}||x||_E\<\frac{\sqrt{2}\lambdaup(2\lambdaup+1)}{\lambdaup-2}e^{2t}||x||_E
\end{align*}
for every $x\in E$ and $\lambdaup>2$. Eventually, we obtain the following exponential bound for the semigroup $\{S(t)\}_{t\geqslant 0}$: 
\[||S(t)x||_E\<\sqrt{2}\inf_{\lambdaup\in(2,\infty)}\frac{\lambdaup(2\lambdaup+1)}{\lambdaup-2}\,e^{2t}||x||_E=\sqrt{2}(4\sqrt{5}+9)e^{2t}||x||_E.\]
This semigroup is also equincontinuous, since \[||S(t)-S(\tau)||_{\mathcal L}=\sup_{||(\mathring{u}_1,\mathring{u}_2)||_E\<1}||(S(t)-S(\tau))(\mathring{u}_1,\mathring{u}_2)||_E\<\int_\tau^t||w(s)||_2\,ds+||w(t)-w(\tau)||_2.\] \par\noindent Define $F_1\colon I\map L^2(\R{n})$ and $F_2\colon I\times L^2(\R{n})\map L^2(\R{n})$ by the formulae
\[F_1(t):=\f\{v\in L^2(\R{n})\colon h^1_1\f(t,x,\int_{\R{n}}\!k_1^1(t,y)\,dy\g)\<v(x)\<h^1_2\f(t,x,\int_{\R{n}}\!k_2^1(t,y)\,dy\g)\text{ a.e. on }\R{n}\g\}\!.\]
and
\begin{align*}
&F_2(t,u):=\\&\f\{v\in L^2(\R{n})\colon h^2_1\f(t,x,\int_{\R{n}}\!k_1^2(t,y)u(y)\,dy\g)\<v(x)\<h^2_2\f(t,x,\int_{\R{n}}\!k_2^2(t,y)u(y)\,dy\g)\text{ a.e. on }\R{n}\g\}\!.
\end{align*}
Let $F\colon I\times E\map E$ be a map given by $F(t,(u_1,u_2)):=F_1(t)\times F_2(t,u_2)$. Assume that the mapping $F$ forms a multivalued perturbation of the abstract semilinear integro-differential inclusion \eqref{P}. To be able to apply Corollary \ref{solsetcor}. we need to verify conditions $(\A_2)$ and $(\F_1)$-$(\F_5)$. As far as condition $(\A_2)$ is concerned, we have verified it above. Hypotheses $(\F_1)$ and $(\F_2)$ follow immediately from assumption $(\text{h}_1)$.\par Take $(u_1,u_2)\in E$ and $(f_1,f_2)\in F(t,(u_1,u_2))$. Then 
\begin{equation*}
\begin{dcases}
|f_1(x)|\<\max\f\{\f|h_1^1\f(t,x,\int_{\R{n}}k_1^1(t,y)\,dy\g)\g|,\f|h_2^1\f(t,x,\int_{\R{n}}k_2^1(t,y)\,dy\g)\g|\g\}\<c_1(t,x)&\\
|f_2(x)|\<\max\f\{\f|h_1^2\f(t,x,\int_{\R{n}}k_1^2(t,y)u_2(y)\,dy\g)\g|,\f|h_2^2\f(t,x,\int_{\R{n}}k_2^2(t,y)u_2(y)\,dy\g)\g|\g\}\<c_2(t,x,||u_2||_2)&\\
\end{dcases}
\end{equation*}
and 
\begin{equation*}
\begin{dcases}
||f_1||_2\<b_1(t)&\text{a.e. on }I\\
||f_2||_2\<b_2(t)(1+||u_2||_2)&\text{a.e. on }I.\\
\end{dcases}
\end{equation*}
Whence \[||F(t,(u_1,u_2))||^+_2\<b_1(t)+b_2(t)(1+||u_2||_2)\<(b_1(t)+b_2(t))(1+||(u_1,u_2)||_E).\] In other words, $F$ satisfies the sublinear growth condition \eqref{mu}. Notice also that the multimap $F(t,\cdot)$ is completely continuous (a.e. on $I$), i.e. it maps bounded sets into relatively weakly compact sets (remember that $L^2(\R{n})$ is reflexive).\par It remains to give reason for assumption $(\F_3)$. Assume that $(u_1^k,u_2^k)\xrightharpoonup[k\to\infty]{E}(u_1,u_2)$ and $(f_1^k,f_2^k)\xrightharpoonup[k\to\infty]{E}(f_1,f_2)$ with $(f_1^k,f_2^k)\in F(t,(u_1^k,u_2^k))$ for $k\geqslant 1$. Observe that for $k\geqslant 1$ \[f_2^k(x)\in\f[h_1^2\f(t,x,\int_{\R{n}}k_1^2(t,y)u_2^k(y)\,dy\g),h_2^2\f(t,x,\int_{\R{n}}k_2^2(t,y)u_2^k(y)\,dy\g)\g]\;\;\text{a.e. on }\R{n}\] and \[z_j^k:=\int_{\R{n}}k_j^2(t,y)u_2^k(y)\,dy\xrightarrow[k\to\infty]{}z_j:=\int_{\R{n}}k_j^2(t,y)u_2(y)\,dy\;\;\text{for a.a. }t\in I\text{ and }j=1,2.\]  Whence \[\overline{\co}\bigcup_{m=k}^\infty\f\{f_2^m(x)\g\}\subset\f[\inf_{m\geqslant k}h_1^2\f(t,x,z^k_1\g),\sup_{m\geqslant k}h_2^2\f(t,x,z^k_2\g)\g]\] and, by $(\text{h}_2)$, \[\bigcap_{k=1}^\infty\overline{\co}\bigcup_{m=k}^\infty\f\{f_2^m(x)\g\}\!\subset\f[\sup_{k\geqslant 1}\inf_{m\geqslant k}h_1^2\f(t,x,z^k_1\g),\inf_{k\geqslant 1}\sup_{m\geqslant k}h_2^2\f(t,x,z^k_2\g)\g]\!\subset\f[h_1^2\f(t,x,z_1\g),h_2^2\f(t,x,z_2\g)\g]\] for a.a. $t\in I$, for a.a. $x\in\R{n}$. Since $f_2(x)\in\bigcap\limits_{k=1}^\infty\overline{\co}\bigcup\limits_{m=k}^\infty\f\{f_2^m(x)\g\}$ a.e. on $\R{n}$ (cf. Corollary \ref{convth}.), we get $f_2\in F_2(t,u_2)$. Notice that $F_1(t)$ is a weakly closed subset of $L^2(\R{n})$. Hence, $f_1\in F_1(t)$. Consequently, the graph of $F(t,\cdot)$ is sequentially closed in $(E,w)\times(E,w)$ for a.a. $t\in I$.\par Owing to Corollary \ref{solsetcor}, we gain confidence that the set $S_{\!F}(\mathring{u}_1,\mathring{u}_2)$ of all integrated solutions to the problem
\begin{equation}\label{cauchy}
\begin{gathered}
\dot{u}(t)\in Au(t)+F\f(t,\int_0^tu(s)\,ds\g)\;\;\;\text{on }I,\\
u(0)=(\mathring{u}_1,\mathring{u}_2).
\end{gathered}
\end{equation}
forms an $R_\delta$ subset of the space $C(I,E)$ furnished with the weak topology. Consider a projection $\Pi\colon S_{\!F}(\mathring{u}_1,\mathring{u}_2)\to C(I,L^2(\R{n}))$, given by $\Pi(u_1,u_2):=u_2$. A short glimpse at the definition of an integrated solution to the Cauchy problem \eqref{cauchy} leads to conclusion that $\mathring{u}_1+\Pi\f(S_{\!F}(\mathring{u}_1,\mathring{u}_2)\g)={\mathcal S}(\mathring{u}_1,\mathring{u}_2)$ (see \cite[Section 7.]{thieme} for clues). One easily sees that the mapping $\w{\Pi}\colon\f(S_{\!F}(\mathring{u}_1,\mathring{u}_2),w\g)\to\f(\Pi\f(S_{\!F}(\mathring{u}_1,\mathring{u}_2\g),w\g)$ is continuous, surjective and proper. Moreover, a careful look at the set \[\f\{(u_1,u_2)\in S_{\!F}(\mathring{u}_1,\mathring{u}_2)\colon u_2=v\g\}\] reveals that it is essentially an $R_\delta$-type set. In practice, this means that the fiber $\Pi^{-1}(\{v\})$ is an acyclic subset of the space $(C(I,E),w)$. Therefore, $\w{\Pi}$ is a Vietoris mapping and $\w{H}^*\f(\f(S_{\!F}\f(\mathring{u}_1,\mathring{u}_2\g),w\g)\g)\approx\w{H}^*\f(\f(\Pi\f(S_{\!F}\f(\mathring{u}_1,\mathring{u}_2\g)\g),w\g)\g)$ (in view of the Vietoris-Begle mapping theorem for Alexander-Spanier cohomology functor \cite[Th.6.9.15]{spanier}). Clearly, the solution set ${\mathcal S}(\mathring{u}_1,\mathring{u}_2)$ must be an acyclic subset of $(C(I,L^2(\R{n})),w)$.
\end{proof}

Let us consider the following initial boundary value problem defined on $I\times\R{}$:
\begin{equation}\label{cauchy2}
\begin{dcases*}
\frac{\partial}{\partial t}u(t,x)-\sum_{j=0}^ka_j D^j u(t,x)=U(t)u(t,\cdot)(x)+h(t,x)&in $I\times\R{}$\\
u(0,x)=\mathring{u}(x)&on $\R{}$\\
||h(t,\cdot)||_2\<r(t,u(t,\cdot))&on $I$.
\end{dcases*}
\end{equation}
\par Let $E$ denote the complex Hilbert space $L^2(\R{},\mathbb{C})$. Our hypotheses on the mappings $r\colon I\times E\to\R{}_+$ and $U(t)\colon E\to E$ are the following:
\begin{itemize}
\item[$(\text{U}_1)$] $U(t)$ is a linear bounded operator for every $t\in I$, $U(\cdot)v$ is measurable for every $v\in E$ and $||U(\cdot)||_{\mathcal L}\in L^1(I)$,
\item[$(\text{r}_1)$] the function $r(\cdot,u)$ is measurable for any $u\in E$,
\item[$(\text{r}_2)$] the function $r(t,\cdot)$ is weakly usc for $t\in I$, 
\item[$(\text{r}_3)$] $r(t,u)\<b(t)(1+||u||_2)$ a.e. on $I$ with $b\in L^1(I)$.
\end{itemize}
\par By the weak solution of the problem \eqref{cauchy2} we mean a function $u\in C(I,E)$ such that $\langle u(\cdot),v\rangle$ is differentiable for every $v\in H^k(\R{})$ and $u$ satisfies
\begin{equation*}
\begin{dcases*}
\frac{d}{dt}\,\langle u(t),v\rangle=\langle\mathring{u},v\rangle+\f\langle u(t),\sum_{j=0}^ka_j D^j v\g\rangle+\f\langle\int_0^tU(s)u(s)+h(s)\,ds,v\g\rangle&for $t\in I$\\
u(0)=0
\end{dcases*}
\end{equation*}
for some function $h\in L^1(I,E)$ such that $||h(t)||_2\<r(t,u(t))$ on $I$.\par Define the polynomial $p(x):=\sum\limits_{j=0}^ka_j(ix)^j$ ($i=$ imaginary unit). Let $a_0,\ldots,a_k\in\mathbb{C}$ and $\omega:=\max\{0,\sup\limits_{x\in\R{}}\RE(p(x))\}$. Fix a constant $L_0>1$ such that $|p(x)|>|a_kx^k|/2$ for all $|x|\geqslant L_0$. Put 
\begin{equation}\label{M}
M:=\f(\frac{32(L_0)^{-2k+1}}{(2k-1)|a_k|^2}+2L_0T^2+4T^2\f(\frac{k(k+1)R}{|a_k|}\g)^2L_0^{-1}+4L_0T^2\sup_{|x|\<L_0}\frac{|p'(x)|}{|p(x)|}\g)^\frac{1}{2},
\end{equation}
where $R:=\max\limits_{1\<j\<k}|a_j|$.
\begin{theorem}
Assume that hypotheses $(\text{U}_1)$ and $(\text{r}_1)$-$(\text{r}_3)$ are satisfied. Suppose that $a_k\neq0$ and $a_j(-i)^{3j}\in\R{}$ for $j=0,\ldots,k$. If $\sup\limits_{x\in\R{}}\RE(p(x))<\infty$, then for every $\mathring{u}\in L^2(\R{})$ the set ${\mathcal S}(\mathring{u})$ of weak solutions to problem \eqref{cauchy2} forms an $R_\delta$ subset of the space $C(I,L^2(\R{},\mathbb{C}))$ endowed with the weak topology.
\end{theorem}

\begin{proof}
Consider the differential operator $A\colon D(A)\to E$ given by $Af:=\sum\limits_{j=0}^ka_jD^jf$, defined on \[D(A):=\f\{f\in E\colon \sum_{j=0}^ka_jD^jf\in E\;\text{distributionally}\g\}.\] Since $a_k\neq 0$, $D(A)=H^k(\R{})$ (\cite[Th.10.14.]{weidmann}). Assumption $a_j(-i)^{3j}\in\R{}$ for $j=0,\ldots,k$ means that the differential operator $A$ is self-adjoint on $E$ (cf. \cite[Th.10.12.]{weidmann}). In view of \cite[Th.4.1.]{hieber} the operator $A$ generates a norm continuous integrated semigroup $\{S(t)\}_{t\geqslant 0}$ on the space $E$, given by \[S(t)f:=\frac{1}{\sqrt{2\pi}}\w{\phi_t}\ast f,\] where $\phi_t(x):=\int_0^te^{p(x)s}\,ds$ and $\sim$ denotes the inverse of the Fourier transformation. Easy calculations show that
\begin{align*}
||\phi_t||_2^2&\<\int\limits_{-\infty}^{-L_0}\frac{|e^{p(x)t}-1|^2}{|p(x)|^2}\,dx+\int\limits_{-L_0}^{L_0}|e^{p(x)t}t|^2\,dx+\int\limits_{L_0}^\infty\frac{|e^{p(x)t}-1|^2}{|p(x)|^2}\,dx\\&\<\int\limits_{-\infty}^{-L_0}\frac{16e^{2\omega t}}{|a_kx^k|^2}\,dx+\int\limits_{-L_0}^{L_0}t^2e^{2\omega t}\,dx+\int\limits_{L_0}^\infty\frac{16e^{2\omega t}}{|a_kx^k|^2}\,dx=\f(\frac{32(L_0)^{-2k+1}}{(2k-1)|a_k|^2}+2L_0t^2\g)e^{2\omega t}.
\end{align*}
For $|x|\geqslant L_0$ we have
\begin{align*}
\frac{|p'(x)|}{|p(x)|}&=\frac{|\sum_{j=1}^kja_ji^jx^{j-1}|}{|p(x)|}\<2\frac{\sum_{j=1}^kj|a_j||x^{j-1}|}{|a_kx^k|}=2\sum_{j=1}^k\frac{j|a_j|}{|a_k||x^{k-j+1}|}\<\frac{2R}{|a_k|}\sum_{j=1}^k\frac{j}{|x|}\\&=\frac{k(k+1)R}{|a_k||x|}.
\end{align*}
Whence
\begin{align*}
\f\Arrowvert\frac{d}{dx}\phi_t\g\Arrowvert_2^2&\<2\f(\,\int\limits_{-\infty}^\infty\f|\frac{p'(x)}{p(x)}te^{p(x)t}\g|^2\,dx+\int\limits_{-\infty}^\infty\f|\frac{p'(x)}{p(x)}\phi_t(x)\g|^2\,dx\g)\\&\<2\f(\,\int\limits_{|x|\geqslant L_0}\frac{t^2e^{2\omega t}(k(k+1)R)^2}{|a_k|^2|x|^2}\,dx+\int\limits_{-L_0}^{L_0}\f|\frac{p'(x)}{p(x)}te^{p(x)t}\g|^2\,dx\g)\\&\<4\f(t^2\f(\frac{k(k+1)R}{|a_k|}\g)^2L_0^{-1}+L_0t^2\sup_{|x|\<L_0}\frac{|p'(x)|}{|p(x)|}\g)e^{2\omega t}.
\end{align*}
Eventually,
\begin{equation}\label{M2}
||\phi_t||_{1,2}\<\f(\frac{32(L_0)^{-2k+1}}{(2k-1)|a_k|^2}+2L_0t^2+4t^2\f(\frac{k(k+1)R}{|a_k|}\g)^2L_0^{-1}+4L_0t^2\sup_{|x|\<L_0}\frac{|p'(x)|}{|p(x)|}\g)^\frac{1}{2}e^{\omega t}.
\end{equation}
Applying \cite[Lem.4.4]{hieber}, \eqref{M} and \eqref{M2} we obtain the following exponential bound for our semigroup:
\begin{equation}\label{S}
||S(t)f||_2\<\frac{1}{\sqrt{2\pi}}||\w{\phi_t}||_1||f||_2\<||\phi_t||_{1,2}||f||_2\<Me^{\omega t}||f||_2.
\end{equation}
As a result, assumption $(\A_2)$ is met.\par Define a multimap $F\colon I\times E\map E$ by the formula \[F(t,u):=U(t)u+\f\{v\in E\colon ||v||_2\<r(t,u)\g\}.\] From $(\text{U}_1)$ and $(\text{r}_1)$ it follows straightforwardly that $F$ satisfies $(\F_1)$-$(\F_2)$. Moreover, \[||F(t,u)||_2^+\<||U(t)||_{\mathcal L}||u||_2+r(t,u)\<(b(t)+||U(t)||_{\mathcal L})(1+||u||_2)\;\;\text{ a.e. on }I,\] i.e. \eqref{mu} holds. The set $\bigcup\limits_{u\in\Omega}\f\{v\in E\colon ||v||_2\<r(t,u)\g\}$ is relatively weakly compact for a.a. $t\in I$ and for any bounded $\Omega\subset E$, since $E$ is reflexive. Whence \[\beta(F(t,\Omega))\<\beta(U(t)\Omega)+\beta\f(\bigcup\limits_{u\in\Omega}\f\{v\in E\colon ||v||_2\<r(t,u)\g\}\g)\<||U(t)||_{\mathcal L}\beta(\Omega)\;\;\text{a.e. on }I.\] \par As it comes to condition $(\F_3)$, let us assume that $u_k\xrightharpoonup[n\to\infty]{E}u$ and $g_k\xrightharpoonup[n\to\infty]{E}g$, where $g_k\in F(t,u_k)$ for $\K$. Suppose that $g_k=U(t)u_k+f_k$. Observe that \[f_k=g_k-U(t)u_k\xrightharpoonup[k\to\infty]{E}g-U(t)u\] and $||g-U(t)u||_2\<\liminf\limits_{k\to\infty}||f_k||_2\<\limsup\limits_{k\to\infty}r(t,u_k)\<r(t,u)$, by $(\text{r}_2)$. Therefore, the weak limit point $g=U(t)u+g-U(t)u\in F(t,u)$. \par By virtue of Corollary \ref{solsetcor}, we know that the set $S_{\!F}(\mathring{u})$ of all integrated solutions of the Cauchy problem \eqref{P} is nonempty $R_\delta$ in the space $C(I,E)$ endowed with the weak topology. One easily sees that ${\mathcal S}(\mathring{u})=S_{\!F}(\mathring{u})$.
\end{proof}

\begin{remark}
In view of \cite[Theorem 4.4.1]{vrabie} the operator $A$, defined as above, generates a $C_0$-semigroup of isometries. The mild solution to the problem \eqref{cauchy2} for $\mathring{u}\in L^2(\R{})=\overline{D(A)}$ coincides with the solution in the sense of Da Prato-Sinestrari i.e., with such a function $u\in C(I,E)$ that  
\begin{equation}\label{prato}
\begin{dcases*}
u(t)=\mathring{u}+A\int_0^tu(s)\,ds+\int_0^tU(s)\int_0^su(\tau)\,d\tau+h(s)\,ds&on $t\in I$\\
||h(t)||_2\<r\f(t,\int_0^tu(s)\,ds\g)&on $I$
\end{dcases*}
\end{equation}
$($cf \cite[Proposition 12.4]{prato}$)$. On the other hand, the mild solution can be thought of as the weak solution in the sense of Ball to the initial boundary value problem \eqref{cauchy2} i.e. a continuous map $u\colon I\to E$ such that for every $v\in D(A^*)$, $\langle u(\cdot),v\rangle$ is absolutely continuous and  
\begin{equation*}
\begin{dcases*}
\frac{d}{dt}\,\langle u(t),v\rangle=\f\langle u(t),\sum_{j=0}^ka_jD^j v\g\rangle+\f\langle U(t)\int_0^tu(s)\,ds+h(t),v\g\rangle&a.e. on $I$\\
u(0)=\mathring{u}\\
||h(t)||_2\<r\f(t,\int_0^tu(s)\,ds\g)&on $I$.
\end{dcases*}
\end{equation*}
\end{remark}\mbox{}\\
\par Let $\Omega\subset\R{N}$ be an open bounded subset with regular boundary $\Gamma$. Consider the following non-densely defined semilinear feedback control system:
\begin{equation}\label{feedback2}
\begin{dcases*}
\frac{\partial x}{\partial t}-\Delta x=U(t)x(t,\cdot)(z)\cdot u(t,z)&a.e. on $I\times\Omega$\\
\res{x}{I\times\Gamma}=0,\;x(0,z)=x_0(z)&a.e. on $\Omega$\\
u(t,z)\in U(t,z,x(t,z))&a.e. on $I\times\Omega,\;\hat{u}\in L^1(I,C(\overline{\Omega}))$
\end{dcases*}
\end{equation}
where $\hat{u}(t):=u(t,\cdot)$. \par The feedback set-valued map $U\colon I\times\overline{\Omega}\times\R{}\map\R{}$ meets the conditions:
\begin{itemize}
\item[$(\uu_1)$] the map $U$ has nonempty closed convex values,
\item[$(\uu_2)$] the map $U(\cdot,\cdot,x(\cdot))$ is ${\mathcal L}(I)\otimes{\mathcal B}(\overline{\Omega})$-measurable for every $x\in C(\overline{\Omega})$,
\item[$(\uu_3)$] $||U(t,z,x)||^+\<\xi(t)(1+|z|)$ a.e. on $I$ for any $(z,x)\in\overline{\Omega}\times\R{}$, where $\xi\in L^1(I,\R{}_+)$.
\item[$(\uu_4)$] \[\underset{\underset{\scriptstyle{y_2\in U(t,z_2,x_2)}}{y_1\in U(t,z_1,x_1)}}{\sup}|y_1-y_2|\<k(t)|z_1-z_2|\] for $z_1,z_2\in\overline{\Omega}$, $x_1,x_2\in\R{}$ and for a.a. $t\in I$ with $k\in L^1(I,\R{}_+)$.
\end{itemize}
By an integrated solution $x$ of the problem \eqref{feedback2} we mean a function $x\in C(I,C(\overline{\Omega}))$ such that \[x(t)=tx_0+\Delta\int_0^tx(s)\,ds+\int_0^t(t-s)\,U(s)x(s)\cdot u(s)\,ds,\;\text{ on }I\] with $u\in L^1(I,C(\overline{\Omega}))$ such that $u(t)(z)\in U(t,z,x(t)(z))$ a.e. on $I\times\Omega$.
\begin{theorem}
Assume that $\{U(t)\}_{t\in I}\subset{\mathcal L}(C(\overline{\Omega}))$, $U(\cdot)x$ is measurable for every $x\in C(\overline{\Omega})$ and $||U(\cdot)||_{\mathcal L}\in L^\infty(I)$. Under conditions $(\uu_1)$-$(\uu_4)$ %and 
%\begin{equation}\label{f4}
%\f\Arrowvert\,||U(\cdot)||_{\mathcal L}\xi\,\g\Arrowvert_1<\f(1+||\Omega||^+\g)^{-1}
%\end{equation}
the set ${\mathcal S}(x_0)$ of integrated solutions of the feedback control system \eqref{feedback2} is $R_\delta$ for every $x_0\in C(\overline{\Omega})$.
\end{theorem}

\begin{proof}
Put $E:=C(\overline{\Omega})$, $C_0(\overline{\Omega}):=\f\{u\in E\colon u=0\text{ on }\Gamma\g\}$ and $Au:=\Delta u$ with \[D(A):=\f\{u\in C_0(\overline{\Omega})\colon\Delta u\in E\text{ distributionally}\g\}.\] In view of \cite[Proposition 14.6]{prato} $A$ generates a contraction analytic semigroup \[e^{At}=\frac{1}{2\pi i}\int\limits_{+C}e^{\lambdaup t}R(\lambdaup,A)\,d\lambdaup\]
on $E$, where $+C$ is a suitable oriented path in the complex plane. Observe that $\overline{D(A)}=C_0(\overline{\Omega})\subsetneq E$. From \cite[Theorem 10.2]{prato} we know that \[R(\lambdaup,A)x=\int\limits_0^{+\infty}e^{-\lambdaup t}e^{At}x\,dt\] for each $x\in E$ and $\lambdaup\in\mathbb{C}$ with $\operatorname{Re}\lambdaup>0$. The formula $S(t):=\int_0^te^{A\tau}\,d\tau$ for $t\geqslant 0$, defines a strongly continuous exponentially bounded family $\{S(t)\}_{t\geqslant 0}\subset{\mathcal L}(E)$ such that $S(0)=0$. In view of \cite[Theorem 10.1]{prato}, the semigroup $\{S(t)\}_{t\geqslant 0}$ is non-degenerate. Clearly, $(0,\infty)\subset\rho(A)$ and $R(\lambdaup,A)=\lambdaup\int_0^\infty e^{-\lambdaup t}S(t)\,dt$. In other words, operator $A$ is the generator of an equicontinuous integrated semigroup $\{S(t)\}_{t\geqslant 0}$.\par Since $U(t,\cdot,x(\cdot))$ is lower semicontinuous,$U(\cdot,\cdot,x(\cdot))$ is ${\mathcal L}(I)\otimes{\mathcal B}(\overline{\Omega})$-measurable and $(I,{\mathfrak L}(I),\ell)$ is $m$-projective, the map $U(\cdot,\cdot,x(\cdot))$ admits a Carath\'eodory-Castaing representation $\z{u_n^x}$ i.e., $U(t,z,x(z))=\overline{\x{u_n^x(t,z)}}$. Let $F\colon I\times E\map E$ be such that $F(t,x):=U(t)x\cdot\co\f(\x{u_n^x(t,\cdot)}\g)$, with "$\cdot$" being the multiplication in the ring $C(\overline{\Omega})$. Since the maps $U(\cdot)\colon I\to E$ and $I\ni t\mapsto u_n^x(t,\cdot)\subset E$ are measurable, the set-valued map $t\mapsto \x{U(t)x\cdot u_n^x(t,\cdot)}$ is measurable. Whence, the convex envelope map $t\mapsto\co\{\x{U(t)x\cdot u_n^x(t,\cdot)}$ is measurable as well. Eventually, the map $F(\cdot,x)$ satisfies assumption $(\F_2)$.\par From $(\uu_4)$ follows that the family $\bigcup\limits_{x\in E}\x{u_n^x(t,\cdot)}$ is equicontinuous. Combined with $(\uu_3)$ it entails the compactness of the map $\bigcup\limits_{x\in M}\x{u_n^x(t,\cdot)}\colon E\map E$ for a.a. $t\in I$. It is routine to check that $\beta(K\cdot M)\<||K||^+\beta(M)$ for $K\subset E$ relatively compact and $M\subset E$ bounded.  Thus, 
\begin{align*}
\beta(F(t,M))&\<\beta\f(\co\f(U(t)M\cdot\f(\bigcup\limits_{x\in M}\x{u_n^x(t,\cdot)}\g)\g)\g)\<\f\Arrowvert\bigcup\limits_{x\in M}\x{u_n^x(t,\cdot)}\g\Arrowvert^+||U(t)||_{\mathcal L}\beta(M)\\&\<\xi(t)\f(1+||\Omega||^+\g)||U(t)||_{\mathcal L}\beta(M),
\end{align*}
a.e. on $I$, by $(\uu_3)$. Hence, $(\F_5)$ holds with $\eta:=\xi\f(1+||\Omega||^+\g)||U(\cdot)||_{\mathcal L}$. \par Since $U(t,\cdot,\cdot)$ is upper hemicontinuous, the sequentiall upper semicontinuity of the map $F(t,\cdot)\colon(E,w)\map(E,w)$ is a straightforward consequence of the Riesz-Markov representation theorem and the Convergence Theorem.\par Observe that \[||F(t,x)||^+\<||U(t)||_{\mathcal L}||x||\,\f\Arrowvert\x{u_n^x(t,\cdot)}\g\Arrowvert^+\,dt\<\f(1+||\Omega||^+\g)||U(t)||_{\mathcal L}\xi(t)(1+||x||),
\]
by $(\uu_3)$. Hence, \eqref{mu} is met.\par We may rewrite equivalently the feedback control problem \eqref{feedback2} as \eqref{P} with $A$ and $F$ as above. Since the assumptions of Corollary \ref{solsetcor} are satisfied, the thesis follows.
\end{proof}

%\section*{Acknowledgement}
%I would like to express my appreciation to the anonymous reviewer for the exceptional and rarely seen reliability and professional %objectivity.

\end{document}